\documentclass[onefignum,onetabnum]{siamart220329}

\usepackage{amssymb}
\usepackage{graphicx}
\usepackage{algorithmic}
\usepackage[caption=false]{subfig}
\usepackage{multirow}
\usepackage{microtype}
\allowdisplaybreaks[4]

\newcommand{\bsmallmatrix}[1]{\begin{bmatrix}\begin{smallmatrix}
#1\end{smallmatrix}\end{bmatrix}}

\newsiamremark{remark}{Remark}
\newsiamremark{hypothesis}{Hypothesis}
\crefname{hypothesis}{Hypothesis}{Hypotheses}
\newsiamthm{claim}{Claim}
\ifpdf
  \DeclareGraphicsExtensions{.eps,.pdf,.png,.jpg}
\else
  \DeclareGraphicsExtensions{.eps}
\fi

\headers{A CJ-FEAST GSVDsolver for the partial GSVD computation}{Z. JIA AND K. Zhang}
\title{A CJ-FEAST GSVDsolver for computing a partial GSVD of a large matrix pair with the generalized singular values in a given interval\thanks{\funding{Supported in part by the National Natural Science Foundation of China
(No.12171273)}}}

\author{Zhongxiao Jia\thanks{Department of Mathematical Sciences, Tsinghua
University, 100084 Beijing, China
  (\email{jiazx@tsinghua.edu.cn}).}
\and Kailiang Zhang\thanks{Department of Mathematical Sciences, Tsinghua
University, 100084 Beijing, China
  (\email{zkl18@tsinghua.org.cn}).}}
\usepackage{amsopn}
\DeclareMathOperator{\diag}{diag}

\begin{document}

\maketitle

% REQUIRED
\begin{abstract}
We propose a CJ-FEAST GSVDsolver to compute a partial
generalized singular value decomposition (GSVD) of a large matrix pair
$(A,B)$ with the generalized singular values in a given interval.
The solver is a highly nontrivial extension of the FEAST
eigensolver for the (generalized) eigenvalue problem and CJ-FEAST SVDsolver
for the SVD problem.
For a partial GSVD problem, given three left and right searching
subspaces, we propose a general
projection method that works on $(A,B)$ {\em directly}, and computes
approximations to the desired GSVD components. For the concerning GSVD
problem, we exploit the Chebyshev--Jackson (CJ) series to
construct an approximate spectral projector of the generalized eigenvalue
problem of the matrix pair $(A^TA,B^TB)$ associated
with the generalized singular values of interest, and use
subspace iteration on it to generate a  right subspace. Premultiplying
it with $A$ and $B$ constructs two left subspaces. Applying the
general projection method to the subspaces constructed leads to the CJ-FEAST
GSVDsolver. We derive accuracy estimates for the approximate spectral
projector and its eigenvalues, and establish a number of convergence
results on the underlying subspaces
and the approximate GSVD components obtained by
the CJ-FEAST GSVDsolver. We propose general-purpose choice strategies for the
series degree and subspace dimension.
Numerical experiments illustrate the efficiency of the CJ-FEAST GSVDsolver.
\end{abstract}

% REQUIRED

% REQUIRED

\section{Introduction}
The generalized singular value decomposition (GSVD) of a matrix pair $(A,B)$
was introduced by Van Loan \cite{van1976generalizing} and developed by Paige and Saunders
\cite{paige1981towards}. It is a generalization of the
singular value decomposition (SVD) of a single matrix to a matrix pair,
and has been a standard matrix decomposition
\cite{bjorck1996numerical,golub2013matrix,stewart1998matrix,stewart2001matrix}.
The GSVD not only provides us with an important mathematical tool but also plays a key role in
a number of disciplines and scientific computing,
e.g., linear discrete ill-posed problems with general-form regularization
 \cite{hansen1998rank}, weighted least squares and total least squares
problems \cite{bjorck1996numerical}, information retrieval, linear discriminant
analysis, and many others \cite{vanhfuffel2013total}.

Let $A\in\mathbb{R}^{m_1\times n}$ and $B\in\mathbb{R}^{m_2\times n}$ with $m_1+m_2\geq n$,
and assume that the stacked matrix $[A^T,B^T]^T$ has full column rank with
the superscript $T$ the transpose of a matrix,
i.e., the null space intersection $\mathcal{N}(A) \cap \mathcal{N}(B)=\{\textbf{0}\}$.
Then the pair $(A,B)$ is called {\em regular}, the same terminology
as for the generalized eigenvalue problem \cite{stewart1990matrix}.
Denote $q_1=\dim(\mathcal{N}(A))$ and $q_2=\dim(\mathcal{N}(B))$,  $q=n-q_1-q_2$,
and $l_1=\dim(\mathcal{N}(A^T))$ and $l_2=\dim(\mathcal{N}(B^T))$.
Then the GSVD of $(A,B)$ is
\begin{equation}\label{gsvd}
    \begin{cases}
        U^TAX=\Sigma_A=\diag\{C_A,\mathbf{0}_{l_1, q_1},I_{q_2}\}, \\
        V^TBX=\Sigma_B=\diag\{S_B,I_{q_1},\mathbf{0}_{l_2, q_2}\},
    \end{cases}
\end{equation}
where $X=[X_q,X_{q_1},X_{q_2}]$ is nonsingular,
$U=[U_q,U_{l_1},U_{q_2}]$ and $V=[V_q,V_{q_1},V_{l_2}]$  are orthogonal,
and the diagonal $C_A=\diag\{c_1,\dots,c_q\}$ and
$S_B=\diag\{s_1,\dots,s_q\}$ satisfy
\begin{equation*}
  0<c_i,s_i<1 \quad\mbox{and}\quad
 c_i^2+s_i^2=1,  \quad i=1,\dots,q;
\end{equation*}
see \cite[p. 309]{golub2013matrix}.
In order to distinguish the block submatrices of $X,U,V$,
we have denoted by the subscripts their column numbers,
by $I_{k}$ and $\mathbf{0}_{k,l}$ the identity matrix of order $k$ and zero
matrix of order $k\times l$, respectively, with the subscript
$k$ dropped whenever it is clear from the context,
and by $e_i$ the $i$th column of $I$.

We call the pairs $\{c_i,s_i\}$ or
$\sigma_i=c_i/s_i,\ i=1,2,\ldots,q$ the {\em nontrivial}
generalized singular values of $(A,B)$ and the corresponding columns $u_i,v_i,x_i$
of $U_q,V_q$ and $X_q$ the left and right generalized singular
vectors.
The partitioned matrix $[X_{q_1}, X_{q_2}]$ corresponds to the {\em trivial}
$q_1$ multiple zero and $q_2$ multiple infinite generalized singular values,
written as $\{c,s\}=\{0,1\}$ and $\{c,s\}=\{1,0\}$, respectively.
We will order the $q$ generalized singular values $\{c_i,s_i\}$ later.
It follows
from \eqref{gsvd} that $X^T(A^TA+B^TB)X=I_n$; that is, $X$ is
$(A^TA+B^TB)$-orthogonal, and its columns are $(A^TA+B^TB)$-orthonormal.
If both $A$ and $B$ are of full column rank, then $q=n$.

In this paper, we are concerned with the computation of
a partial GSVD of a large regular $(A,B)$ with all the generalized singular values
$\sigma_i$ in a given interval. In terms of
the one-to-one correspondence of $c$ and $\sigma$ with
subscripts $i$ dropped, this amounts to the problem:
Given an interval $[c_{\min},c_{\max}]\subset (0,1)$, we want to
compute all the $n_{gv}$ generalized singular quintuples $(c,s,u,v,x)$,
called the GSVD components, with $c\in [c_{\min},c_{\max}]$. Keep in mind
that, in terms of GSVD \eqref{gsvd}, the GSVD components can
be written in the matrix-vector form:
\begin{equation}\label{gsvdvector}
    \begin{cases}
        Ax=c u,   \\
        Bx=s v, \\
        s A^Tu=c B^Tv.
    \end{cases}
\end{equation}

Unlike a partial SVD of a single large matrix,
not much work has been done for the partial GSVD computation
of the large matrix pair
$(A,B)$. Zha \cite{zha1996computing} proposes a joint bidiagonalization (JBD)
method to compute a few {\em extreme} GSVD components of $(A,B)$.
The method is based on a JBD process that successively reduces
$(A,B)$ to a sequence of upper bidiagonal pairs, from which
approximate GSVD components are computed. Kilmer, Hansen and Espa\~{n}ol
\cite{kilmer2007projection} adapt
the JBD process to the linear discrete ill-posed problem with general-form
regularization, and develop a variant of the JBD process that reduces
$(A,B)$ to lower-upper bidiagonal forms. Based on it, they propose
a hybrid projection method. Jia and
Yang~\cite{JiaYang2020} present a new JBD process based iterative algorithm
for the ill-posed problem, which is simpler and cheaper than
and computes a regularized solution at least as accurately as the hybrid
method in \cite{kilmer2007projection}.
They have considered the convergence of extreme generalized singular values.
Jia and Li \cite{jiali2023} have made a detailed numerical analysis on the JBD
method for the GSVD computation and shown that it suffices to maintain
the semi-orthogonality of the four sets of computed Lanczos-type vectors
in finite precision arithmetic when computing generalized
singular values accurately,
where the semi-orthogonality means that two vectors
are numerically orthogonal to the level of $\epsilon_{\mathrm{mach}}^{1/2}$
with $\epsilon_{\mathrm{mach}}$ being the machine precision.
Jia and Li \cite{jia2021joint} propose an effective and efficient partial
reorthogonalization strategy to maintain the desired semi-orthogonality
so as to avoid convergence delay and the
appearance of duplicate approximate generalized singular values. To be practical,
Alvarruiz, Campos and Roman \cite{alvaruiz2022} have recently
developed an explicit thick-restart JBD algorithm for the GSVD computation.

Hochstenbach \cite{hochstenbach2009jacobi} presents a Jacobi--Davidson (JD)
GSVD (JDGSVD) method to compute a number of GSVD components
of $(A,B)$ with the desired generalized
singular values closest to a given target.
The method formulates the GSVD of $(A,B)$ as the
generalized eigendecomposition of the augmented and cross-product matrix pair
$\big(\bsmallmatrix{&A\\A^T&},\bsmallmatrix{I&\\&B^TB}\big)$ or
$\big(\bsmallmatrix{&B\\B^T&},\bsmallmatrix{I&\\&A^TA}\big)$,
which requires that $B$ or $A$ be of full
column rank, computes the relevant eigenpairs,
and recovers the desired GSVD components from the converged eigenpairs.
At each subspace expansion step, for $\big(\bsmallmatrix{&A\\A^T&},\bsmallmatrix{I&\\&B^TB}\big)$,
the method solves an $(m_1+n)$-dimensional correction equation approximately, and
the lower $n$-dimensional and upper $m_1$-dimensional parts of the approximate
solution are used to expand the right searching subspace
and one left searching subspace; similarly,
for $\big(\bsmallmatrix{&B\\B^T&},\bsmallmatrix{I&\\&A^TA}\big)$,
$(m_2+n)$-dimensional correction equation is approximately solved.
Huang and Jia \cite{huangjia2021} have proved that
there is a constant multiple $\kappa(B^TB)=\kappa^2(B)$ or
$\kappa(A^TA)=\kappa^2(A)$ in the error bounds for the
computed eigenvectors, where $\kappa(\cdot)$ denotes the $2$-norm
condition number of a matrix.
Consequently, with $B$ or $A$ ill conditioned, the computed GSVD components
may have poor accuracy,
which has been numerically confirmed \cite{huangjia2021}.
Therefore, the two formulations of generalized eigenvalue problems
are not numerically equivalent in finite precision arithmetic.
One prefers the first formulation
if $\kappa(B)<\kappa(A)$, otherwise takes the second.
In general, the transforming
a GSVD problem to the generalized eigenvalue problem of some
augmented and/or cross-product matrix form is not of generality
due to possible accuracy loss, poor numerical stability and worse conditioning, etc.

Huang and Jia \cite{huangjia2023} have recently proposed
a cross-product free JDGSVD (CPF-JDGSVD) method
to compute several GSVD components of $(A,B)$ with the generalized
singular values closest to a given target. The CPF-JDGSVD method is
$A^TA$ and $B^TB$ free at the extraction stage, and it premultiplies
the right searching subspace by $A$ and $B$ to construct two left ones
separately. The method implicitly realizes the Rayleigh--Ritz projection of the
generalized eigenvalue problem of $(A^TA,B^TB)$ onto a given
right searching subspace.
At the subspace expansion stage, it suffices to iteratively solve an $n$-by-$n$
correction equation with low or modest accuracy,
and its approximate solution is used to expand the searching subspaces.
The method is numerically backward stable and can compute GSVD components
much more accurately than JDGSVD in \cite{hochstenbach2009jacobi}.
Furthermore, to overcome the possible irregular and slow convergence of
CPF-JDGSVD when the desired generalized singular values are interior,
Huang and Jia \cite{huangjia2022twoharmonic} have proposed two harmonic
extraction based JDGSVD-type methods, which have better convergence behavior
and are more efficient than CPF-JDGSVD.

Over the past two decades, the Sakurai--Sugiura (SS) method
\cite{sakurai2003projection}
and the FEAST eigensolver \cite{polizzi2009density} have been proposed
for computing the eigenpairs of a large matrix or matrix pair with
the eigenvalues in a given region.
The SS method and its variants SS--RR and SS-Arnoldi methods
have been intensively studied in, e.g.,
\cite{ikegami2010filter,imakura2014SSarnoldi,sakurai2007cirr},
and the z-Pares package has been developed \cite{Futamura2014Online}.
The SS--RR and SS-Arnoldi methods and the FEAST eigensolver
both construct a good approximate spectral projector
associated with all the eigenvalues in the region of interest.
These methods are based on the contour integral formula
of the exact spectral projector, and exploit some numerical quadrature to
implicitly construct an approximate spectral projector and explicitly
compute its product with vectors \cite{polizzi2009density,sakurai2003projection}, which involves
solutions of large shifted linear systems.
Other rational filtering approaches can be found in
\cite{guttel2015zolotarev,Kolling2021rationalspectral,Ruhe1998rationalkrylov}.
These two methods are Krylov or block Krylov subspace
based methods when the contour is a circle and the trapezoidal rule
is used, and they realize the Rayleigh--Ritz projection and compute
Ritz approximations \cite{sakurai2016}.
The FEAST eigensolver is a subspace iteration based method, and it
generates a sequence of subspaces of fixed dimension,
onto which the Rayleigh--Ritz projection of the original matrix or matrix pair
is realized and the Ritz approximations are computed.
The solver has been intensively investigated in, e.g., \cite{guttel2015zolotarev,kestyn2016feast,tang2014feast},
and the FEAST numerical library has been available \cite{polizzi2020feast}.

In the contour integral-based FEAST eigensolver, one needs to solve
several large shifted linear systems at each iteration.
These linear systems are typically indefinite and some of them
could be ill conditioned, which is definitely the case if
a quadrature node is close to some eigenvalue,
so that Krylov subspace iterative methods, e.g.,
the MINRES or GMRES method \cite{saad2003iterative}, may be excessively
slow. On the other hand, in finite precision arithmetic, as is shown
theoretically and confirmed numerically in \cite{jia2022afeastsvdsolver,jia2023augmentedCJFEAST},
the contour integral-based FEAST eigensolver and
SVDsolver have intrinsic deficiencies: (i) the ultimately attainable
accuracy of approximate solutions of some of these linear systems may not
fulfill the convergence requirement of the solver, and (ii) even if it
converges mathematically, the achievable residual norms of approximate
eigenpairs or singular triplets cannot drop below a reasonable
stopping tolerance in finite precision arithmetic.

Recently, Imakura and Sakurai \cite{imakura2021complexmoment}
have adapted the SS method to the SVD problem and
proposed an algorithm that can compute the singular values in a given interval.
The authors of the current paper have extended the FEAST eigensolver
to the SVD problem and proposed two CJ-FEAST SVDsolvers that can compute the
singular values in a given interval \cite{jia2022afeastsvdsolver,jia2023augmentedCJFEAST}.
Each of these two SVDsolvers has its own merit and disadvantage, and one can
easily make a proper choice between them, depending on sizes of the desire singular values, as has been shown in \cite{jia2023augmentedCJFEAST}.
A distinctive feature of CJ-FEAST SVDsolvers
is that rather than exploiting the contour
integral formula and any numerical quadrature or rational filtering,
we construct an approximate spectral projector
with arbitrary accuracy by the Chebyshev--Jackson (CJ) series expansion
\cite{jay1999electronic,rivlin1981introduction} without
solving any shifted linear system. We have quantitatively established sharp
pointwise error bounds for the approximation of the CJ series to a
specific step function whose values 1, $0.5$ and $0$
correspond to the eigenvalues of the underlying exact spectral
projector. Based on them, we have derived the accuracy estimate of the
constructed approximate spectral
projector, analyzed its eigenvalue distribution,
established a number of convergence results on the CJ-FEAST SVDsolvers,
and proposed reliable and practical strategies for the series degree and
subspace dimension. The solvers are numerically illustrated to be
much more efficient and also more robust
than the corresponding contour integral-based solvers. Remarkably, as has been
addressed in \cite{jia2022afeastsvdsolver,jia2023augmentedCJFEAST},
the CJ-FEAST SVDsolver is directly applicable to the real
symmetric or Hermitian eigenvalue problem.

The GSVD has several fundamental distinctions with the SVD and
generalized eigendecomposition of a matrix pair. For instance,
(i) the GSVD itself is considerably more complicated
than the SVD, where each GSVD component is a
quintuple other than a triple, and there is
one {\em non}-orthogonal
right generalized singular vector matrix $X$ and {\em two} orthogonal left
generalized singular vector matrices $U$ and $V$; (ii) there may be trivial
zero and infinite generalized singular values;
(iii) the underlying spectral projector and its approximations
are {\em no longer} symmetric;
(iv) the GSVD is also different from the generalized eigendecomposition,
which provides information on generalized eigenpairs or
the generalized eigenvalues and pairs of left and right generalized
eigenvectors. As a result,
the partial GSVD computation under consideration is substantially more
difficult and complicated than solving standard or generalized
eigenvalue problems and SVD problems.
Among others, there are two difficult ingredients that
must be provided: (i) construct {\em three} reasonable
left and right searching subspaces; (ii) propose a reasonable
projection method that works on $(A,B)$ {\em directly} without involving any
cross-product or augmented matrix, so that the extracted approximate
GSVD approximations converge if the searching subspaces contain
sufficiently accurate approximations to the desired generalized
singular vectors.

It is known from \eqref{gsvdvector} that
the generalized singular values $(c,s)$
and right generalized singular vectors $x$ of the pair $(A,B)$ are the
generalized eigenpairs of the cross-product pair $(A^TA,B^TB)$:
$s^2A^TAx=c^2B^TBx$. Therefore, with the nonzero $c,s$ and $x$ available,
the left generalized singular vectors $u=Ax/s$ and $v=Bx/c$.
In principle, the SS-type methods and the FEAST eigensolver can be
straightforwardly adapted to the mathematically equivalent
generalized eigenvalue problem $(A^TA,B^TB)$ with
the singular values with $c\in [c_{\min},c_{\max}]$ and $x$.
However, this involves the cross-product matrices $A^TA$ and $B^TB$,
and is thus unattractive and out of favor due to numerous reasons,
e.g., the worse conditioning of the generalized eigenvalue problem, the
possible accuracy loss of the computed approximations,
the loss of numerical orthogonality of approximate left
generalized singular vectors recovered from the {\em converged}
approximate right generalized singular vectors when $c$ or $s$ is small.

In this paper, for a partial GSVD problem, we propose a general
projection method that works on $A$ and $B$ directly for given
left and right searching subspaces, and the method
becomes the standard extraction approach
proposed in \cite{huangjia2023} for two specially chosen left
subspaces obtained by premultiplying the right one with $A$ and
$B$, respectively. For the aforementioned GSVD problem,
by exploiting the fact that the GSVD $(A,B)$ is mathematically
equivalent to the generalized
eigendecomposition of the symmetric positive definite (SPD) matrix pair
$(A^TA-B^TB,A^TA+B^TB)$, whose generalized eigenpairs
are $(c^2-s^2,x)$, we construct an approximation to the spectral projector of
this generalized eigenvalue problem with $c\in [c_{\min},c_{\max}]$
by using the CJ series expansion. We then
apply subspace iteration to the approximate spectral projector, and
generate a sequence of approximate right generalized singular
subspaces associated with $c\in [c_{\min},c_{\max}]$.
At extraction phase, we apply the general projection
method to the approximate left and right
subspaces with the two left ones formed by premultiplying the right one
with $A$ and $B$, respectively, leading to  a specific
CJ-FEAST GSVDsolver. Mathematically,
as far as the generalized singular values and right generalized
singular vectors are concerned, we show that the method is mathematically
equivalent to the Rayleigh--Ritz projection of the generalized
eigenvalue problem of $(A^TA,B^TB)$ with respect to the
right subspace.
%We thus call the obtained approximate
%GSVD components the Ritz approximations or quintuples.

As is expected, the convergence analysis of the CJ-FEAST GSVDsolver is
much more involved and complicated than that of the CJ-FEAST SVDsolver, and a
reliable determination of the number $n_{gv}$ of desired GSVD components
with $c\in [c_{\min},c_{\max}]$ and that of the subspace dimension are more
subtle. Based on the theoretical results to be established,
we will propose practical selection strategies for the CJ series degree
$d$ and the subspace dimension $p$. The solver critically requires
$p\geq n_{gv}$; otherwise it fails.
We will establish a number of convergence results on the underlying
approximate subspaces and the Ritz values and vectors, showing
how fast they converge in terms of the accuracy of the approximate
spectral projector and its eigenvalues.
%For the
%algorithmic implementations, one must solve several linear systems with the
%same symmetric positive definite matrix $A^TA+B^TB$ at each iteration.

%The paper is organized as follows.
In \Cref{sec: preliminaries}, we introduce preliminaries, and propose
a class of general projection methods for the partial GSVD problem. Then we
present an algorithmic sketch of the specific CJ-FEAST GSVDsolver.
In \Cref{sec: CJ-FEAST GSVD}, we develop a detailed CJ-FEAST
GSVDsolver, establish estimates for the accuracy and eigenvalues of the
approximate spectral projector, and consider reliable estimation of
$n_{gv}$ and some implementation details.
In \Cref{sec: convergence analysis}, we present the convergence results
on the CJ-FEAST GSVDsolver. \Cref{sec:numerical experiments}
reports numerical experiments to justify theoretical results
and illustrate the performance of the CJ-FEAST GSVDsolver.
In \Cref{sec: conclusion}, we conclude the paper.

\section{Preliminaries and
a class of general projection methods}\label{sec: preliminaries}

\subsection{Preliminaries}

From \eqref{gsvd}, the generalized eigendecomposition of the SPD matrix
pair $(A^TA-B^TB,A^TA+B^TB)$ is
\begin{equation}\label{eigd}
    X^T(A^TA-B^TB)X=\Sigma_{A}^T\Sigma_{A}-\Sigma_{B}^T\Sigma_{B}, \
     X^T(A^TA+B^TB)X=I_n.
\end{equation}
This means that the $q$ nontrivial generalized singular values $\sigma=c/s$ and
associated right generalized singular vectors $x$ of $(A,B)$ correspond to
the eigenvalues $c^2-s^2=2c^2-1$ of $(A^TA-B^TB,A^TA+B^TB)$ and the
associated eigenvectors. Meanwhile,
$(A^TA-B^TB,A^TA+B^TB)$ has $q_1$ multiple eigenvalues $-1$
and $q_2$ multiple eigenvalues $1$, which correspond to the trivial
$q_1$ multiple zero and $q_2$ multiple infinite generalized
singular values of $(A,B)$, respectively. Therefore, the generalized singular
values of $(A,B)$ are mapped to the spectral
interval $[-1,1]$ of  $(A^TA-B^TB,A^TA+B^TB)$.

More generally, the GSVD of $(A,B)$ corresponds to the generalized
eigendecomposition of the parameterized matrix pair
$(\eta A^TA-\xi B^TB,\gamma A^TA+\zeta
B^TB)$ for any {\em nonnegative} $\eta,\xi,\gamma,\zeta$ with the pairs
$\{\eta,\xi\}\neq\{0,0\}, \{\gamma,\zeta\}\neq\{0,0\},
\{\eta,\gamma\}\neq\{0,0\}$ and $\{\xi,\zeta\}\neq\{0,0\}$.
The generalized eigenvalues are the pairs $\{\eta c^2-\xi s^2, \gamma
c^2+\zeta s^2\}=\{(\eta+\xi)c^2-\xi,(\gamma-\zeta)c^2+\zeta\}$,
and the corresponding generalized eigenvectors are $x$.
The $q_1$ multiple zero and $q_2$ multiple infinite generalized singular
values correspond to
the $q_1$ multiple generalized eigenvalues $\{-\xi,\zeta\}$ and
the $q_2$ multiple generalized eigenvalues $\{\eta,\gamma\}$ of the
parameterized matrix pair.
Particularly, the generalized singular values of $(A,B)$ with
$c\in [c_{\min},c_{\max}]\in (0,1)$ are mapped to the generalized eigenvalues
in $[\frac{(\eta+\xi)c_{\min}^2-\xi}{(\gamma-\zeta)c_{\min}^2+\zeta},
\frac{(\eta+\xi)c_{\max}^2-\xi}{(\gamma-\zeta)c_{\max}^2+\zeta}]$.
%When $\eta=\xi=\gamma=\zeta=1$, the generalized singular values
%with $c\in [c_{\min},c_{\max}]$ are mapped to the interval
%$[2c_{\min}^2-1,2c_{\max}^2-1]$.
Notice that $\gamma A^TA+\zeta B^TB$ is unconditionally
positive semi-definite for any nonnegative pair $\{\gamma,\eta\}\neq \{0,0\}$
but the positive definiteness requires that $A$ or $B$ be of full column rank
when $\gamma=1,\zeta=0$ or $\gamma=0,\zeta=1$.
%\textcolor{blue}{We can always choose $\eta,\xi,\gamma,\zeta$ to
%make the spectral interval of the parameterized matrix $\bar{S} = (\gamma %A^TA+\zeta B^TB)^{-1}(\eta A^TA-\xi B^TB)$ be $[-1,1]$.
%Otherwise, a linear function is needed to map the spectral interval of %$\bar{S}$ to $[-1,1]$
%and this linear function is equivalent to modify $\eta$ and $\xi$.
%To ensure this, we may need to estimate the largest and the smallest %generalized singular values of matrix pair $(A, B)$.}
%\end{comment}

From now on, write
\begin{equation}\label{defS}
 H=A^TA+B^TB,\ \   S = H^{-1}(A^TA-B^TB).
\end{equation}
From \eqref{eigd}, we have $X^{-1}=X^TH$, and the
eigendecomposition of $S$ is
\begin{equation}\label{sdecomp}
    S=X(\Sigma_{A}^T\Sigma_{A}-\Sigma_{B}^T\Sigma_{B})X^{-1}
    =X(\Sigma_{A}^T\Sigma_{A}-\Sigma_{B}^T\Sigma_{B})X^TH.
\end{equation}
Therefore, the generalized singular values $\{c,s\}$ of $(A,B)$ correspond to
the eigenvalues $c^2-s^2$ of $S$.
Although we will use the SPD matrix pair $(A^TA-B^TB,H)$,
all the results and analysis can be straightforwardly adapted to
the parameterized pair $(\eta A^TA-\xi B^TB,
\gamma A^TA+\zeta B^TB)$.
The two mathematical changes are (i) to replace $S$ by $\bar{S}=
(\gamma A^TA+\zeta B^TB)^{-1}(\eta A^TA-\xi B^TB)$ in
the context and (ii) to modify the function $l(z)$ in the beginning of
\Cref{sec: CJ-FEAST GSVD} correspondingly.
%We will come back to the parameterized form in \Cref{subsec: implementation}.

Given an interval $[c_{\min},c_{\max}]\subset [0,1]$, suppose we are
interested in all the generalized singular values $\sigma=c/s$ with
$c \in [c_{\min},c_{\max}]$ and/or
the left and right generalized singular vectors $u,v$ and $x$.
Then the desired generalized singular values of $(A,B)$ correspond to
the eigenvalues $2c^2-1\in [2c_{\min}^2-1,2c_{\max}^2-1]$ of $S$.
Define
\begin{equation}\label{ps}
    P_{S}=X_{in}X_{in}^T H +\frac{1}{2}X_{e}X_{e}^TH,
\end{equation}
where $X_{in}$ consists of the columns of $X$ corresponding to the eigenvalues of
$S$ in the open interval $(2c_{\min}^2-1,2c_{\max}^2-1)$ and $X_{e}$ consists
of the columns of $X$ corresponding to the eigenvalues of $S$ that equal
$2c_{\min}^2-1$ or $2c_{\max}^2-1$. The eigenvalues of $P_S$
are $1,\frac{1}{2}$ and 0. We will call $P_{S}$
the spectral projector of $S$ associated with $c\in [c_{\min},c_{\max}]$.

From \eqref{gsvdvector}, the GSVD residual of an approximate generalized
singular quintuple $(\hat{c},\hat{s},
\hat{u},\hat{v},\hat{x})$ with $\hat{c}^2+\hat{s}^2=1$ is
\begin{equation}\label{res}
    r=r(\hat{c},\hat{s},\hat{u},\hat{v},\hat{x})=\begin{bmatrix}
        A\hat{x}-\hat{c} \hat{u}\\
        B\hat{x}-\hat{s} \hat{v}\\
        \hat{s}A^T\hat{u}-\hat{c}B^T\hat{v}
    \end{bmatrix},
\end{equation}
and the size of $\|r\|$ will be used to judge the convergence
of $(\hat{c},\hat{s},\hat{u},\hat{v},\hat{x})$.

We introduce the $H$-inner product of two real vectors $y$ and $z$ by
\begin{equation}\label{hinner}
    \langle y,z\rangle_{H}=y^THz,
\end{equation}
which induces the $H$-norm
\begin{equation}\label{hnorm}
    \|y\|_{H}=\langle y,y\rangle_{H}^{\frac{1}{2}}=(y^T Hy)^{\frac{1}{2}}=\|H^{\frac{1}{2}}y\|
\end{equation}
with $H^{\frac{1}{2}}$ being the square root of $H$,
where $\|\cdot\|$ denotes the vector 2-norm and
the induced matrix norm. We
define the $H$-angle $\angle( y,z )_{H}$ of $y$ and $z$ via
\begin{equation}\label{hangle}
    \cos\angle( y,z )_{H}=\frac{\langle y,z\rangle_{H}}{\|y\|_{H}\|z\|_{H}}, \quad \sin\angle( y,z )_{H}
    =\sqrt{1-\cos^2\angle( y,z )_{H}}.
\end{equation}
We comment that the $H$-inner product, the induced norm and angle are naturally
valid for complex vectors, provided the transpose is replaced by the conjugate
transpose and $H$-orthogonal by $H$-unitary. But the current context involves
only real vectors.

Suppose that $[Z_1,Z_2]$ and $[W_1,W_2]$ are two $n\times n$ $H$-orthogonal
matrices in conforming partitions. Since $H^{\frac{1}{2}}W_1$
and $H^{\frac{1}{2}}Z_2$ are column orthonormal, the 2-norm distance between
the subspaces span$\{Z_1\}$ and span $\{W_1\}$
(cf. \cite[section 2.5.3]{golub2013matrix}) is
\begin{equation}\label{C-distdef}
    {\mathrm{dist}}({\mathrm{span}}\{Z_1\}, {\mathrm{span}}\{W_1\})=\|Z_2^T H W_1\|.
\end{equation}

For any scalars $\lambda$ and $\mu$,
the chordal metric $\chi (\lambda, \mu)$ between them is defined by
\begin{equation*}
    \chi (\lambda, \mu) = |\lambda-\mu|/\sqrt{1+|\lambda|^2}\sqrt{1+|\mu|^2},
\end{equation*}
which is used to measure the error of approximate and exact
eigenvalues of a regular matrix pair \cite[Chapter 15]{parlett1998symmetric}.

\subsection{A class of general projection methods}

For any given left and right searching subspaces $\mathcal{U}$,
$\mathcal{V}$ and $\mathcal{X}$ with
\begin{equation}\label{dim}
\min\{\dim(\mathcal{U}),\dim(\mathcal{V})\}\geq \dim(\mathcal{X}),
\end{equation}
we now propose a class of general projection methods
for the computation of a desired partial GSVD. The method finds
scalars $\hat{c}\geq 0$ and $\hat{s}\geq 0$ with $\hat{c}^2+\hat{s}^2=1$ and
vectors $\tilde u\in\mathcal{U}$, $\tilde v\in\mathcal{V}$, $\hat
x\in\mathcal{X}$ with $\|\tilde u\|=\|\tilde v\|=
\|\hat x^T\|_H=1$ that satisfy the requirements
\begin{equation}\label{pgsvd}
	\left\{\begin{aligned}
		A\hat x-\hat{c}\hat u&\perp\mathcal{U},\\
		B\hat x-\hat{s}\hat v&\perp\mathcal{V},\\
		\hat{s} A^T\hat u-\hat{c} B^T\hat v&\perp\mathcal{X},
	\end{aligned}\right.
\end{equation}
and uses the $(\hat{c},\hat{s},\hat u,\hat v,\hat x)$ as approximations
to some of the GSVD components $(c,s,u,v,x)$ of $(A,B)$.
%Obviously, the method is more complicated than those
%projection methods for (generalized) eigenvalue problems and SVD problems.

Particularly, for a given right searching subspace
$\mathcal{X}$, one natural choice for two left searching subspaces is
\begin{equation}\label{leftsub}
    \mathcal{U}=A\mathcal{X},\ \mathcal{V}=B\mathcal{X},
\end{equation}
which was proposed in \cite{huangjia2022twoharmonic,huangjia2023}
and will be adopted in this paper. The rationale of
\eqref{leftsub} is that $\mathcal{U}$ and $\mathcal{V}$
are equal to the corresponding
left generalized singular subspaces with $A$ and $B$ when
$\mathcal{X}$ is an exact right generalized singular
subspace of $(A,B)$, as commented in
\cite{huangjia2023}. In this case, an informal
justification for the method \eqref{pgsvd} is that
all the $(\hat{c},\hat{s},\hat u,\hat v,\hat x)$
are exact GSVD components
of $(A,B)$. A continuity argument suggests that if
$\mathcal{X}$ is nearly a right generalized singular subspace of $(A,B)$
then $\mathcal{U}$ and $\mathcal{V}$ are nearly left generalized
singular subspaces with $A$ and $B$ and all the
$(\hat{c},\hat{s},\hat u,\hat v,\hat x)$ should be nearly GSVD components of $(A,B)$.

With $\mathcal{U}$ and $\mathcal{V}$ defined as \eqref{leftsub}, we have
\begin{equation*}
    A\hat x-\hat{c}\hat u\in \mathcal{U},\
    B\hat x-\hat{s}\hat v \in \mathcal{V},
\end{equation*}
and the first two orthogonality conditions in \eqref{pgsvd} become
\begin{equation}\label{Ax=cu, Bx=sv}
    A\hat x-\hat{c}\hat u=\mathbf{0},\
    B\hat x-\hat{s}\hat v=\mathbf{0},
\end{equation}
which are the projection requirements in
\cite{huangjia2022twoharmonic,huangjia2023}.

It is straightforward from \eqref{Ax=cu, Bx=sv} that
the $(\hat{c}^2,\hat x)$ are the Ritz approximations of
the standard Rayleigh--Ritz projection of the generalized eigenvalue problem
of the SPD matrix pair $(A^TA,H)$ onto the subspace $\mathcal{X}$
and satisfy
\begin{equation}\label{cpp}
    (A^TA-\hat{c}^2 H)\hat x\perp\mathcal{X}.
\end{equation}
Similarly, let $\hat\sigma=\hat{c}/\hat{s}$. Then
the $(\hat\sigma^2,\hat x)$
are the Ritz approximations of $(A^TA,B^TB)$ with respect to $\mathcal{X}$:
\begin{equation}\label{cpp2}
    (A^TA-\hat\sigma^2 B^TB)\hat x\perp\mathcal{X}.
\end{equation}
Therefore, we call $(\hat{c},\hat{s},\hat u,\hat v,\hat x)$
the Ritz approximations of $(A,B)$ with respect to the given
left and right subspaces, $(\hat{c},\hat{s})$ or
$\hat\sigma$ the Ritz value, and $\hat u$, $\hat v$
and $\hat x$ the left and right Ritz vectors of $(A,B)$.
Such Rayleigh--Ritz projection is also called the standard
extraction approach, as named in \cite{huangjia2022twoharmonic,huangjia2023}
because of \eqref{cpp2}.

Let the columns of $\widetilde{X},\ \widetilde{U}$ and
$\widetilde{V}$ form orthonormal base of $\mathcal{X},\
\mathcal{U}$ and $\mathcal{V}$, and denote $\hat{x}=\widetilde{X}d$,
$\hat{u}=\widetilde{U}e$ and $\hat{v}=\widetilde{V}f$ with
$\|d\|_H=\|e\|=\|f\|=1$. Then
\eqref{pgsvd} amounts to
\begin{equation}\label{sgsvdg}
	\left\{\begin{aligned}
		(\widetilde{U}^TA\widetilde{X})d&=\hat{c} e,\\
		(\widetilde{V}^TB\widetilde{X})d&=\hat{s} f, \\
		\hat{s} (\widetilde{U}^TA\widetilde{X})^Te &=\hat{c} (\widetilde{V}^TB\widetilde{X})^T f.
	\end{aligned}\right.
\end{equation}
Requirement~\eqref{dim} guarantees that the projection method
computes $\dim(\mathcal{X})$ approximate GSVD components.
%Alternatively, if we replace the above $\widetilde{X}$ by a $H$-orthonormal
%one, then the corresponding
%$\|d\|_H=1$ is replaced by $\|d\|=1$. In both cases,
%keep \eqref{gsvdvector} in mind.
\eqref{sgsvdg} shows that $(\hat{c},\hat{s},d,e,f)$ are
the GSVD components of the projection matrix pair
$(\widetilde{U}^TA\widetilde{X},\widetilde{V}^TB\widetilde{X})$,
and we thus need to compute the GSVD of this small matrix pair
to obtain the Ritz approximations.

The JBD method \cite{jiali2023,zha1996computing}
falls into the framework of \eqref{pgsvd}. One benefit is that,
unlike any available projection method working on the generalized
eigenvalue problem of $(A^TA,B^TB)$ or an augmented and
cross-product matrix pair in \cite{hochstenbach2009jacobi},
since the approximations $\hat{u},\hat{v}$ and $\hat{x}$
are computed {\em independently}, they naturally
maintain the desired numerical features, such as numerical
orthogonality and $H$-numerical orthogonality
in finite precision arithmetic, respectively, provided that the GSVD of
the small projection matrix pair is computed accurately
and the orthonormal base of the left and right searching
subspaces are constructed to working precision.

For our GSVD problem, based on the projection
method \eqref{pgsvd}, \Cref{alg:subspace iteration} describes a
framework of the CJ-FEAST GSVDsolver to be studied and developed later,
where $P$ is an approximation to $P_{S}$.  Starting with an initial right
subspace $\mathcal{X}^{(0)}$ with dimension $p\geq n_{gv}$,
subspace iteration on $P$ generates a sequence of
approximate right generalized singular subspaces $\mathcal{X}^{(k)}$, and
the left subspaces are constructed by $\mathcal{U}^{(k)}=A\mathcal{X}^{(k)},
\mathcal{V}^{(k)}=B\mathcal{X}^{(k)}$, $k=0,1,\ldots$.
At each iteration $k$, we realize projection \eqref{pgsvd}
with respect to the left and right subspaces
$\mathcal{U}^{(k)}, \mathcal{V}^{(k)}$ and $\mathcal{X}^{(k)}$,
compute the Ritz approximations
$(\hat{c}^{(k)},\hat{s}^{(k)},\hat{u}^{(k)},\hat{v}^{(k)},\hat{x}^{(k)})$,
and take those for
all the $\hat{c}^{(k)}\in [c_{\min},c_{\max}]$ as approximations of the
desired GSVD components $(c,s,u,v,x)$.
%We call the scalar pairs $(\hat{c}^{(k)},\hat{s}^{(k)})$ or
%$\hat{\sigma}^{(k)}=\hat{c}^{(k)}/\hat{s}^{(k)}$ the Ritz values, and
%$\hat{u}^{(k)},\hat{v}^{(k)}$
%and $\hat{x}^{(k)}$ the left and right Ritz vectors, respectively.
If $P=P_{S}$ defined by \eqref{ps} and the subspace dimension $p=n_{gv}$,
then under the condition that no vector in the initial subspace
$\mathcal{X}^{(0)}$ is $H$-orthogonal to the desired
${\mathrm{span}}\{X_{in},X_{e}\}$,
\Cref{alg:subspace iteration} finds the $n_{gv}$ desired GSVD components
in {\em one} iteration since
$\mathcal{X}^{(1)}={\mathrm{span}}\{X_{in},X_{e}\}$ and $\mathcal{U}^{(1)},\
\mathcal{V}^{(1)}$ are the exact left generalized singular subspaces
of $(A,B)$ associated with all the $c\in [c_{\min},c_{\max}]$.

\begin{algorithm}
    \caption{The basic FEAST GSVDsolver:
    Subspace iteration on the approximate spectral projector
        $P$ for computing the partial GSVD of $(A,B)$.}
    \label{alg:subspace iteration}
    \begin{algorithmic}[1]
        \REQUIRE{The interval $[c_{\min},c_{\max}]$, the approximate
        spectral projector
        $P$, a $p$-dimensional right subspace $\mathcal{X}^{(0)}$
        with $p\geq n_{gv}$ and the left subspaces $\mathcal{U}^{(0)}=A\mathcal{X}^{(0)}$
        and $\mathcal{V}^{(0)}=B\mathcal{X}^{(0)}$.}
        \ENSURE{$n_{gv}$ converged Ritz quintuples
        $(\hat{c}^{(k)},\hat{s}^{(k)},\hat{u}^{(k)},\hat{v}^{(k)},\hat{x}^{(k)})$.}
        \FOR{$k=0,1,\dots,$}
        \STATE{Projection: realize the method \eqref{pgsvd} with respect to
        $\mathcal{U}^{(k)}, \mathcal{V}^{(k)}$ and $\mathcal{X}^{(k)}$,
        and compute the Ritz approximations $(\hat{c}^{(k)},\hat{s}^{(k)},\hat{u}^{(k)},\hat{v}^{(k)},\hat{x}^{(k)})$.}
        \STATE{Compute the norms of residuals \eqref{res} of
        $(\hat{c}^{(k)},\hat{s}^{(k)},\hat{u}^{(k)},\hat{v}^{(k)},\hat{x}^{(k)})$
        for all $\hat{c}^{(k)}\in [c_{\min},c_{\max}]$.}
            \IF{converged}
                \STATE{\textbf{break}}
            \ENDIF
            \STATE{Update the right searching subspace $\mathcal{X}^{(k+1)}=P\mathcal{X}^{(k)}$ and the left searching subspaces $\mathcal{U}^{(k+1)}=A\mathcal{X}^{(k+1)},\mathcal{V}^{(k+1)}=B\mathcal{X}^{(k+1)}$.}
        \ENDFOR
    \end{algorithmic}
\end{algorithm}

\section{A detailed CJ-FEAST GSVDsolver and its
ingredients}\label{sec: CJ-FEAST GSVD}

\subsection{Approximate spectral projector}

Define the function $l(z)=2z^2-1$,
which maps $[0,1]$ to $[-1,1]$ and a given interval
$[c_{\min},c_{\max}]\subset [0,1]$ to $[2c_{\min}^2-1,2c_{\max}^2-1]\subset
[-1,1]$, the spectral interval of $S$.
Define the step function
\begin{equation}\label{hfun}
    h(l(z))=
    \begin{cases}
        1, \qquad z\in (c_{\min}, c_{\max}), \\
        \frac{1}{2}, \qquad z\in \{c_{\min}, c_{\max}\}, \\
        0, \qquad z\in [0,1]\setminus [c_{\min}, c_{\max}],
    \end{cases}
\end{equation}
which corresponds to the eigenvalues of the spectral projector $P_S$
defined by \eqref{ps}. Furthermore,
it follows from \eqref{sdecomp} and \eqref{ps} that the matrix
\begin{equation}\label{hps}
    h(S)=X h(2\Sigma_{A}^T\Sigma_{A}-I)X^TH =P_{S}.
\end{equation}
Thus by the
CJ series expansion $\psi_d(l(z))$ of $h(l(z))$
\cite{jay1999electronic,jia2022afeastsvdsolver,rivlin1981introduction},
we can construct an approximate spectral projector
\begin{equation}\label{peigen}
    P=\psi_{d}(S)=\sum_{j=0}^{d}\rho_{j,d}\eta_jT_{j}(S)
    =X \psi_{d}(2\Sigma_{A}^T\Sigma_{A}-I) X^{-1}.
\end{equation}
Notice that $P$ and $S$ share the same eigenvector matrix $X$,
and the eigenvalues of $P$ are $\gamma_i:=\psi_d(l(c_i)),\ i=1,2,\ldots,q$,
$q_1$ multiple $\gamma:=\psi_d(l(0))$ and $q_2$ multiple
$\gamma:=\psi_d(l(1))$.

Now we analyze the error $\|P_{S}-P\|$ and estimate the eigenvalues
of $P$.
\begin{theorem}\label{Thm:accuracyps}
    Given the interval $[c_{\min},c_{\max}]\subset [0,1]$, let
    \begin{align*}
         & \alpha=\arccos(l(c_{\min})), \quad   \beta=\arccos(l(c_{\max})),                                               \\
         & \Delta_{il}=|\arccos(l(c_{il}))-\alpha|, \quad \Delta_{ir}=|\arccos(l(c_{ir}))-\beta|, \\
         & \Delta_{ol}=|\arccos(l(c_{ol}))-\alpha|, \quad \Delta_{or}=|\arccos(l(c_{or}))-\beta|,
    \end{align*}
    where $c_{il},\ c_{ir}$ and
    $c_{ol},\ c_{or}$ are the diagonal elements of $\Sigma_{A}$ that
    are the closest to $c_{\min}$ and $c_{\max}$ inside and outside
    $[c_{\min},c_{\max}]$, respectively, and
    define
    \begin{equation}\label{deltamin}
        \Delta_{\min}={\min}\{\Delta_{il},\Delta_{ir},\Delta_{ol},\Delta_{or}\}.
    \end{equation}
    Then
    \begin{align}
       % \|h(2\Sigma_{A}^T\Sigma_{A}-I)-\psi_d(2\Sigma_{A}^T\Sigma_{A}-I)\|
       % \leq \frac{\pi^6}{2(d+2)^3\Delta_{\min}^{4}}, \label{Accuracy of diag} %\\
         \|P_{S}-P\| \leq \frac{\pi^6\kappa}{2(d+2)^3\Delta_{\min}^{4}}, \label{Accuracy of projector}
    \end{align}
    where $\kappa=\kappa([A^T,B^T])$ is the 2-norm
    condition number of $[A^T,B^T]^T$.
    Suppose the diagonals of $\Sigma_{A}$ in $[c_{\min},c_{\max}]$ are
    $c_1,\ldots,c_{n_{gv}}$ with $c_1,\ldots,c_r$ in open interval
    $(c_{\min},c_{\max})$ and $c_{r+1},\ldots, c_{n_{gv}}$
    equal to $c_{\min}$ or $c_{\max}$
    and those in $[0,1]\setminus [c_{\min},c_{\max}]$
    are $c_{n_{gv}+1},\ldots,c_n$, and label
    $\gamma_i = \psi_d(l(c_i)),\ i=1,2,\ldots,r$, $i=r+1,\ldots,n_{gv}$ and $i=n_{gv}+1,\dots,n$
    in decreasing order, respectively. Then if
    \begin{equation}\label{dsize}
        d>\frac{\sqrt[3]{2}\pi^2}{\Delta_{\min}^{4/3}}-2,
    \end{equation}
    it holds that
%    \begin{equation}\label{14}
%        \|h(2\Sigma_{A}^T\Sigma_{A}-I)-\psi_d(2\Sigma_{A}^T\Sigma_{A}-I)\| < %\frac{1}{4}
%    \end{equation}
 %   and
    \begin{equation}\label{evhatp}
        1\geq \gamma_1  \geq \cdots \geq \gamma_r > \frac{3}{4}> \gamma_{r+1 }\geq \cdots \geq \gamma_{n_{gv}}
            >\frac{1}{4}>\gamma_{n_{gv}+1}\geq \cdots\geq \gamma_n \geq 0.
    \end{equation}
\end{theorem}

\begin{proof}
    Note that the eigenvalues of
    $P_{S}$ are
    \begin{equation*}
        h(l(c_{i}))=
        \begin{cases}
            1, \quad 1\leq i \leq r,                      \\
            \frac{1}{2}, \quad r+1\leq i\leq n_{gv}, \\
            0, \quad n_{gv}+1\leq i \leq n.
        \end{cases}
    \end{equation*}
    Then we obtain
    \begin{align}
        \|X^{-1}(P_{S}-P)X\|   & =\|X^{-1}(h(S)-\psi_d(S))X \|   \notag \\
                            & =\|h(2\Sigma_{A}^T\Sigma_{A}-I)-\psi_d(2\Sigma_{A}^T\Sigma_{A}-I) \|  \notag\\
                            & =\max_{i=1,2,\ldots,n}|h(l(c_i))-\psi_d(l(c_i))|                            \notag\\
                            & =\max_{i=1,2,\ldots,n}|h(\cos(\theta_i))-\psi_d(\cos(\theta_i))|                        ,\label{ph}
    \end{align}
    where $\theta_i=\arccos(l(c_i))$. Note that
    \begin{equation*}
        \Delta_{\min} \leq {\min}\{2\pi-2\alpha,\alpha-\beta,2\beta\}.
    \end{equation*}
    It then follows from \eqref{ph} and
    the pointwise convergence results on $\psi_d(l(z))$
    \cite[Theorem 3.2, Theorem 3.3]{jia2022afeastsvdsolver} that
    \begin{equation}\label{Accuracy of diag}
        \|X^{-1}(P_{S}-P)X\|
        =\|h(2\Sigma_{A}^T\Sigma_{A}-I)-\psi_d(2\Sigma_{A}^T\Sigma_{A}-I) \| \leq \frac{\pi^6}{2(d+2)^3\Delta_{\min}^{4}}.
    \end{equation}
%    which establishes \eqref{Accuracy of diag}.
    From \cite[Theorem 2.3]{Hansen1989regularization} we have
    \begin{equation}\label{kappax}
        \kappa(X)=\kappa.
    \end{equation}
    Therefore,
    \begin{equation*}
        \|P_{S}-P\| \leq \kappa(X)\|X^{-1}(P_{S}-P)X\|\leq
        \frac{\pi^4\kappa}{(2d+4)^2\Delta_{\min}^2},
    \end{equation*}
    which proves \eqref{Accuracy of projector} and indicates that
    $P=\psi_{d}(S)$ converges to $P_{S}$ as $d$ increases.

    Since \cite[Theorem 3.1]{jia2022afeastsvdsolver} proves that $\gamma_i,\,i=1,2,\ldots,n$ are in $[0,1]$,
    we have
    \begin{equation}\label{hrelation}
        \|h(2\Sigma_{A}^T\Sigma_{A}-I)-\psi_d(2\Sigma_{A}^T\Sigma_{A}-I)\|=
        \max\biggl\{1-\gamma_r,
        \max_{r+1\leq i\leq n_{gv}}\biggl|\frac{1}{2}-\gamma_i\biggr|, \gamma_{n_{gv}+1}\biggr\}.
    \end{equation}
    Under condition \eqref{dsize}, it is easily justified from
    \eqref{Accuracy of diag} that
    \begin{equation}\label{14}
    \|h(2\Sigma_{A}^T\Sigma_{A}-I)-\psi_d(2\Sigma_{A}^T\Sigma_{A}-I)\|<\frac{1}{4}.
    \end{equation}
    Then by \eqref{hrelation} and
    the labeling order of $\gamma_i, i=1,2,\ldots,n$ we have \eqref{evhatp}.
\end{proof}

\begin{remark}\label{rem:error}
    As $d$ increases, $\gamma_i\approx 1,\ i=1,2,\ldots,r$,
    $\gamma_i\approx \frac{1}{2},\ i=r+1,\ldots, n_{gv}$,
    and $\gamma_i\approx 0,\ i = n_{gv}+1,\ldots,n$.
    In fact, by \eqref{Accuracy of diag}, we can make $\|h(2\Sigma_{A}^T\Sigma_{A}-I)-\psi_d(2\Sigma_{A}^T\Sigma_{A}-I)\|<\epsilon$ with
    $\epsilon$ arbitrarily small by increasing $d$. In this case, we have
    \begin{align*}
        1-\epsilon<&\gamma_i\leq 1, i = 1,2,\dots,r,\\
        \frac{1}{2}-\epsilon<&\gamma_i<\frac{1}{2}+\epsilon, i = r+1,\dots,n_{gv}, \\
        0\leq&\gamma_i<\epsilon,\ i = n_{gv}+1,\ldots,n.
    \end{align*}
\end{remark}

\begin{remark}
    $\gamma_i,i=1,2,\ldots,n_{gv}$ are the first $n_{gv}$ dominant eigenvalues
    of $P$, and ${\mathrm{span}}\{X_{in},X_{e}\}$ is the corresponding
    dominant eigenspace,
    provided that the series degree $d$ is sufficiently big such that
    \eqref{14} holds. Particularly,
    if none of $2c_{\min}^2-1$ and $2c_{\max}^2-1$ is an eigenvalue of $S$,
    then the first dominant eigenvalues $\gamma_1,\ldots,\gamma_{n_{gv}}$ of $P$ correspond to the desired generalized singular values of $(A,B)$,
    provided that
    $$
    \|h(2\Sigma_{A}^T\Sigma_{A}-I)-\psi_d(2\Sigma_{A}^T\Sigma_{A}-I)\|<\frac{1}{2}.
    $$
In either case, given an initial subspace of dimension $n_{gv}$
    that no vector in it is $H$-orthogonal to ${\mathrm{span}}\{X_{in},X_{e}\}$,,
    the sequence of subspaces generated by applying subspace iteration to $P$
    converges to ${\mathrm{span}}\{X_{in},X_{e}\}$.
    Therefore, the CJ-FEAST GSVDsolver should compute the desired $n_{gv}$ GSVD
    components of $(A,B)$ successfully.
\end{remark}

\begin{remark}
    The convergence rate $1/(d+2)^3$ in bounds
    \eqref{Accuracy of projector} and \eqref{Accuracy of diag} is the sharpest possible and cannot be improved,
    as has been numerically illustrated in \cite{jia2022afeastsvdsolver}.
\end{remark}

\subsection{Determination of the subspace dimension}

Note that the trace
${\mathrm{tr}}(P_S)=r+\frac{n_{gv}-r}{2}=\frac{r+n_{gv}}{2}$,
which equals $n_{gv}$ when none of $2c_{\min}^2-1$ and $2c_{\max}^2-1$
is an eigenvalue of $S$. A good estimate for ${\mathrm{tr}}(P)$
enables us to reliably choose the subspace dimension $p$ to ensure
$p\geq n_{gv}$, which is necessary for \Cref{alg:subspace iteration}.

The following lemma \cite{avron2011randomized,Cortinovis2021onrandom}
is on estimates for the trace of a symmetric matrix based on the Monte--Carlo
simulation.

\begin{lemma}\label{lem:stochastic estimation}
    Let $G$ be an $n\times n$ real symmetric matrix, and
    $H_{M}=\frac{1}{M}\sum_{i=1}^{M}z_i^T G z_i$,
    where the components $z_{ij}$ of the random vectors $z_i$ are independent
    and identically distributed Rademacher random variables, i.e.,
    ${\mathrm{Pr}}(z_{ij}=1)={\mathrm{Pr}}(z_{ij}=-1)=\frac{1}{2}$.
    Then the expectation ${\mathrm{E}}(H_M)={\mathrm{tr}}(G)$, variance
    ${\mathrm{Var}}(H_M)=\frac{2}{M}\bigl(\|G\|_F^2-\sum_{i=1}^{n}G_{ii}^2\bigr)$,
    and probability
    ${\mathrm{Pr}}(|H_M-{\mathrm{tr}}(G)|\geq \epsilon )\leq \delta$ for
    $ M \geq 8\epsilon^{-2}(\|G\|_F^2+2\epsilon \|G\|)\ln(\frac{2}{\delta})$
    with $\delta<1$,
    where $\|\cdot\|_F$ denotes the Frobenius norm of a matrix.
\end{lemma}

In our context, unlike that for the SVD or standard symmetric (Hermitian)
eigenvalue problem, $P_{S}$ defined by \eqref{ps} and its approximation
$P$ constructed by the CJ series expansion are {\em no longer} symmetric.
To estimate the trace of a real unsymmetric matrix $P$,
we can estimate ${\mathrm{tr}}(P+P^T)/2={\mathrm{tr}}(P)$, as suggested
in \cite{Cortinovis2021onrandom}.
As is shown in \cite[Theorem 3.1]{jia2022afeastsvdsolver},
the eigenvalues $\gamma_i,\,i=1,2,\ldots,n$ of our $P$ are unconditionally
{\em nonnegative} and lie in $[0,1]$, meaning that ${\mathrm{tr}}(P)>0$.
We can prove the following theorem,
which establishes a {\em relative} error estimate for ${\mathrm{tr}}(P)$.

\begin{theorem}\label{Thm:stochastic estimation}
    For ${\mathrm{tr}}(P)>0$, ${\mathrm{Pr}}(|H_M-{\mathrm{tr}}(P)|\geq
    \epsilon \ {\mathrm{tr}}(P))\leq \delta$ when
    $ M \geq 8 \epsilon^{-2}(\|P\|_F^2+2\|P\| \epsilon \ {\mathrm{tr}}(P))
    \ln(\frac{2}{\delta})/{\mathrm{tr}}(P)^2$.
\end{theorem}

\begin{proof}
   Since ${\mathrm{tr}}(P)>0$ and ${\mathrm{tr}}(P+P^T)/2={\mathrm{tr}}(P)$,
for a given $\epsilon$,   we replace $G$ by $(P+P^T)/2$ and $\epsilon$ by
$\epsilon \ {\mathrm{tr}}(P)$
    in \cref{lem:stochastic estimation}, so that
    \begin{equation*}
        {\mathrm{Pr}}(|H_M-{\mathrm{tr}}(P)|\geq \epsilon \
        {\mathrm{tr}}(P))\leq \delta
    \end{equation*}
holds when
    \begin{equation*}
        M \geq 8\biggl(\bigl\|(P+P^T)/2\bigr\|_F^2(\epsilon \
        {\mathrm{tr}}(P))^{-2}+
        \bigl\|P+P^T\bigr\| (\epsilon \
        {\mathrm{tr}}(P))^{-1}\biggr)\ln(\frac{2}{\delta}).
    \end{equation*}
    It then follows from
    \begin{align*}
        &\frac{1}{2}\|P+P^T\|_F \leq \frac{1}{2}(\|P\|_F + \|P^T\|_F)
        = \|P\|_F
    \end{align*}
    and
    \begin{align*}
        &\frac{1}{2}\|P+P^T\| \leq  \|P\|
    \end{align*}
    that the assertion holds.
\end{proof}

\begin{theorem}
Let $\Delta_{\min}$ be defined as \eqref{deltamin}. Then the trace
${\mathrm{tr}}(P)$ satisfies
    \begin{equation}\label{estp}
       |{\mathrm{tr}}(P_{S})-{\mathrm{tr}}(P)|\leq
       \frac{n\pi^6}{2(d+2)^3\Delta_{\min}^{4}}.
    \end{equation}
\end{theorem}

\begin{proof}
    By \eqref{hfun}--\eqref{peigen}, we have
    \begin{align*}
        |{\mathrm{tr}}(P_{S})- {\mathrm{tr}}(P)| &=\left|\sum_{i=1}^n(h(l(c_i))-\gamma_i)\right| \\
                &\leq \sum_{i=1}^n|h(l(c_i))-\gamma_i| \\
                &\leq n \max_{i=1,2,\ldots,n}|h(l(c_i))-\gamma_i|\\
                &= n\|h(2\Sigma_{A}^T\Sigma_{A}-I)-\psi_d(2\Sigma_{A}^T\Sigma_{A}-I) \|,
    \end{align*}
    which, from \eqref{Accuracy of diag}, proves \eqref{estp}.
\end{proof}

The error estimate \eqref{estp} is similar to that for the CJ-FEAST SVDsolver
in \cite{jia2022afeastsvdsolver}, and the same analysis
is directly applicable to \eqref{estp}, which states that
the large factor $n$ essentially behaves like a
constant $\mathcal{O}(1)$.

Now we analyze the sample number $M$ needed for reliably
estimating ${\mathrm{tr}}(P)$.
It is straightforward from \eqref{peigen}, \eqref{evhatp} and
\eqref{kappax} that
\begin{equation*}
    \|P\| \leq \kappa(X)=\kappa.
\end{equation*}
Since the Frobenius norm and the 2-norm are consistent \cite[p.29]{stewart2001matrix},
by \eqref{evhatp} we have
\begin{equation*}
    \|P\|_F \leq \|X\| \|X^{-1}\| \|\psi_d(2\Sigma_A^T\Sigma_A-I)\|_F \approx \sqrt{n_{gv}} \kappa.
\end{equation*}
We remark that if $P$ is replaced by $P_S$ then $\|P_S\|_F\leq \sqrt{n_{gv}}
\kappa$ rigorously.
Substitute the above relation into the condition on $M$ in
\Cref{Thm:stochastic estimation}. Then the condition on $M$ approximately
becomes
\begin{equation}\label{samplenumber}
    M \geq  8\epsilon^{-2} \frac{\kappa
    (\kappa+2\epsilon)}{n_{gv}}\ln\bigl(\frac{2}{\delta}\bigr).
\end{equation}
Notice that a reasonably small $\epsilon \in [0.01,0.1]$ enables
$H_M$ to be a reliable estimate of
${\mathrm{tr}}(P)$ with the probability $1-\delta
\approx 1$ for $\delta \sim 10^{-2}$. Suppose the size of $\kappa(X)$ is
modest. Then \eqref{samplenumber} shows that a modestly sized
$M$, say $20\sim 30$, can give a
reliable estimate; the bigger $n_{gv}$ is, the smaller
$M$ is required. As a result, once $P$ is a reasonably
good approximation to $P_S$, a modest $M$ suffices to
estimate ${\mathrm{tr}}(P_S)$.

In applications, in order to guarantee the subspace dimension $p\geq n_{gv}$,
we adopt the strategy proposed in \cite{jia2022afeastsvdsolver}, and choose
\begin{equation}\label{pchoice}
    p=\lceil \mu H_M\rceil, \ \mu \geq 1.1, \ M \in [20,30],
\end{equation}
where $\lceil \cdot \rceil$ is the ceil function.
We describe the procedure as \Cref{alg:subspace dimension},
and refer to \cite{jia2022afeastsvdsolver} for details on Step 1 in it.
%A lot of numerical experiments will demonstrate that this selection strategy
%ensures $p\geq n_{gv}$.

\begin{algorithm}
    \caption{Determination of the subspace dimension $p$}
    \label{alg:subspace dimension}
    \begin{algorithmic}[1]
        \REQUIRE{The matrix pair $(A,B)$, the interval $[c_{\min},c_{\max}]$,
        the series degree $d$, and $M$ Rademacher random $n$-vectors
        $z_1,z_2,\dots,z_M$.}
        \ENSURE{$p$.}
        \STATE{Compute the CJ coefficients $\eta_j$ and $\rho_{j,d}$,
        $j=0,1,\ldots,d$.}
        \STATE{Compute $H_{M}=\frac{1}{M}\sum_{i=1}^{M}z_i^T P z_i
        =\frac{1}{M}\sum_{i=1}^{M}\sum_{j=0}^d \rho_{j,d}\eta_j z_i^T
        T_j(S)z_i$.}
        \STATE{Choose $p$ in terms of \eqref{pchoice}}.
    \end{algorithmic}
\end{algorithm}

\subsection{The CJ-FEAST GSVDsolver}
Having determined the approximate spectral projector $P$ by \eqref{peigen}
and the subspace dimension $p$ by \Cref{alg:subspace dimension}, for a given
right subspace $\mathcal{X}^{(0)}$ of dimension $p$, the application
\Cref{alg:subspace iteration} to $P$ generates
a sequence of approximate eigenspaces $\mathcal{X}^{(k)}$ of $P$
associated with its $p$ dominant eigenvalues
$\gamma_i,\ i=1,2,\ldots,p$, and form the left subspaces $\mathcal{U}^{(k)}=A\mathcal{X}^{(k)},
\mathcal{V}^{(k)}=B\mathcal{X}^{(k)}$. At iteration $k$, we
project $(A,B)$ onto the three subspaces to compute the Ritz approximations
$(\hat{c}^{(k)},\hat{s}^{(k)},\hat{u}^{(k)},\hat{v}^{(k)},\hat{x}^{(k)})$.

We now describe some implementation details. For $\mathcal{X}^{(0)}$ spanned by the columns
of $X^{(0)}\in\mathbb{R}^{n\times p}$ with full column rank,
which is generated randomly in a normal distribution.
%compute the orthonormal basis matrix $Q^{(0)}$ of $\mathcal{X}^{(0)}$
%by the thin QR factorization $X^{(0)}=Q^{(0)}R^{(0)}$.
Then at iteration $k=0,1,\dots$, %generate the right searching subspace
%\textcolor{blue}{$\mathcal{X}^{(k)}={\mathrm{span}}\{X^{(k)}=PQ^{(k-1)}\}$ with $k=0$ skipped,
compute the thin QR factorizations
\begin{equation}\label{qr}
    X^{(k)}=Q^{(k)}R^{(k)},\quad AQ^{(k)}=Q_1^{(k)}\bar{A}^{(k)}, \quad BQ^{(k)}=Q_2^{(k)}\bar{B}^{(k)},
\end{equation}
where $R^{(k)},\bar{A}^{(k)},\bar{B}^{(k)}\in \mathbb{R}^{p\times p}$ are upper triangular.
Then the columns of $Q^{(k)}$ form an orthonormal basis of $\mathcal{X}^{(k)}$.
Suppose that $AQ^{(k)}$ and $BQ^{(k)}$ are of full column rank.
Then the columns of $Q_1^{(k)}$ and $Q_2^{(k)}$ form orthonormal bases of
$\mathcal{U}^{(k)}$ and $\mathcal{V}^{(k)}$, and
\begin{equation}\label{proj}
    (Q_1^{(k)})^TAQ^{(k)}=\bar{A}^{(k)}, \quad (Q_2^{(k)})^TBQ^{(k)}=\bar{B}^{(k)}
\end{equation}
form the projection matrix pair $(\bar{A}^{(k)},\bar{B}^{(k)})$.
We compute the approximate GSVD component matrices $(\hat{\Sigma}_{A}^{(k)},
\hat{\Sigma}_{B}^{(k)}, \hat{U}^{(k)}, \hat{V}^{(k)}, \hat{X}^{(k)})$.
If not converged, generate the right subspace
$\mathcal{X}^{(k+1)}={\mathrm{span}}\{X^{(k+1)}=P\hat{X}^{(k)}\}$.
We describe the procedure as \Cref{alg:PGSVD},
which is called the cross-product based CJ-FEAST GSVDsolver as we make use of the approximation $P$ of
the spectral projector $P_S$ of the cross-product matrix pair $(A^TA-B^TB,A^TA+B^TB)$.

\begin{algorithm}
    \caption{The cross-product based CJ-FEAST GSVDsolver}
    \label{alg:PGSVD}
    \begin{algorithmic}[1]
        \REQUIRE{The matrix pair $(A, B)$, the interval $[c_{\min}, c_{\max}]$,
        the series degree $d$, a full column rank matrix $X^{(0)}\in \mathbb{R}^{n\times p}$ with $p\geq n_{gv}$. }
        \ENSURE{The $n_{gv}$ converged Ritz approximations
        $(\hat{c}_i^{(k)},\hat{s}_i^{(k)},\hat{u}_i^{(k)},\hat{v}_i^{(k)},\hat{x}_i^{(k)})$
        with $\hat{c}_i^{(k)}\in [c_{\min},
        c_{\max}]$.}
        \STATE{Compute the CJ coefficients $\eta_j$ and $\rho_{j,d}$, $j=0,1,\ldots,d$.}
        \FOR{$k=0,1,\dots,$}
            \STATE{Compute the QR factorizations \eqref{qr}.}
            \STATE{Compute the GSVD of the projection matrix pairs $(\bar{A}^{(k)},\bar{B}^{(k)})$:
            \begin{equation*}
                \bar{A}^{(k)}=\bar{U}^{(k)}\hat\Sigma_{A}^{(k)}(\bar{X}^{(k)})^{-1},\
                \bar{B}^{(k)}=\bar{V}^{(k)}\hat\Sigma_{B}^{(k)}(\bar{X}^{(k)})^{-1}
            \end{equation*}
            with
            $\hat{\Sigma}_{A}^{(k)}=\diag(\hat{c}_1^{(k)},\ldots,\hat{c}_p^{(k)})$
            and
            $\hat{\Sigma}_{B}^{(k)}=\diag(\hat{s}_1^{(k)},\ldots,\hat{s}_p^{(k)})$.}
            \STATE{Form the right and left Ritz approximation matrices:
            \begin{equation*}
                \hat{X}^{(k)}=Q^{(k)}\bar{X}^{(k)},\ \hat{U}^{(k)}=Q_{1}^{(k)}\bar{U}^{(k)},\
                \hat{V}^{(k)}=Q_{2}^{(k)}\bar{V}^{(k)},
            \end{equation*}
            where $\hat{X}^{(k)}$ is $H$-orthonormal and $\hat{U}^{(k)}$
            and $\hat{V}^{(k)}$ are orthonormal.}
            \STATE{Select those $\hat{c}_i^{(k)}\in [c_{\min}, c_{\max}]$, and compute the norms of residuals \eqref{res} of Ritz approximations $(\hat{c}_i^{(k)},\hat{s}_i^{(k)},
            \hat{u}_i^{(k)},\hat{v}_i^{(k)},\hat{x}_i^{(k)})$ with
            $\hat{u}_i^{(k)}=\hat{U}^{(k)}e_i,\hat{v}_i^{(k)}
            =\hat{V}^{(k)}e_i,\hat{x}_i^{(k)}=\hat{X}^{(k)}e_i$,
            where $e_i$ is the $i$-th column of identity matrix $I_p$.}
            \IF{converged}
                \STATE{\textbf{break}}
            \ENDIF
            \STATE{Compute
            $X^{(k+1)}=P\hat{X}^{(k)}=\sum_{j=0}^d \rho_{j,d}\eta_j T_j(S)\hat{X}^{(k)}$.}% and the QR factorizations \eqref{qr}.}
        \ENDFOR
    \end{algorithmic}
\end{algorithm}

\subsection{More implementation details}\label{subsec: implementation}

We do not need to explicitly form the dense approximate spectral projector
$P$, which is exploited only implicitly in \Cref{alg:subspace dimension} and
\Cref{alg:PGSVD}. The most consuming part of \Cref{alg:PGSVD}
is to form matrix-vector (matrix) products with $P$. Next we analyze the
costs of $P$ in \Cref{alg:subspace dimension} and \Cref{alg:PGSVD}.

Recall the expression \eqref{peigen} of $P$. For a given vector $z$,
exploiting the three-term recurrence of Chebyshev polynomials, we have
\begin{equation*}
    T_{0}(S)z=z,\quad  T_{1}(S)z=Sz, \quad
    T_{j+1}(S)z=2 ST_{j}(S)z-T_{j-1}(S)z.
\end{equation*}
We compute $T_{j}(S)z, j=1,2,\dots,d$ recursively, and form $Pz$.

For the CJ series of degree $d$, it is seen from
$S=(A^TA+B^TB)^{-1}(A^TA-B^TB)$ that forming $Pz$ needs
to compute the $d$ matrix-vector products $Sq_i,\ i=1,2,\ldots,d$ for certain
$q_i$, which amounts to the solutions of the $d$ linear systems
\begin{equation}\label{linearsystemtemp}
    (A^TA+B^TB)y_i=(A^TA-B^TB)q_i.
\end{equation}
They can be solved by either a direct solver \cite{golub2013matrix}, e.g.,
by computing the
sparse Cholesky factorization of $A^TA+B^TB$, or an iterative solver, e.g., the
(preconditioned) CG method \cite{saad2003iterative}.
Alternatively, these linear systems amount to the least squares problems
\begin{equation}\label{leastsquaretemp}
    \min_{y\in\mathbb{R}^n} \biggl\|\begin{bmatrix}
        A\\
        B
    \end{bmatrix}y-\begin{bmatrix}
        A\\
        -B
    \end{bmatrix}q_i\biggr\|,
\end{equation}
which are generally better conditioned than the linear systems in
\eqref{linearsystemtemp},
so that we can make use of a direct solver, e.g.,
the sparse QR factorization, or an iterative
solver, e.g., the LSQR algorithm \cite{bjorck1996numerical}, to solve them more
accurately. We remark that
the CG method for \eqref{linearsystemtemp} is the CGLS or CGNR method
by noticing that \eqref{linearsystemtemp} is the normal equation
of \eqref{leastsquaretemp}. Therefore, the CG method for
\eqref{linearsystemtemp} is mathematically equivalent to the LSQR algorithm
for \eqref{leastsquaretemp} when the initial guess for $y_i$ in CG is zero vector and the starting vector in LSQR is $[A^T,-B^T]^Tq_i$.
The convergence rates of the CG and LSQR algorithms critically rely on the
size of $\kappa^2([A^T,B^T])$, and they converge
fast when $[A^T,B^T]^T$ is well conditioned \cite{bjorck1996numerical}.

Recall \Cref{sec: preliminaries}. More generally, in order to compute $Pz$
for a given vector $z$, we
can solve the SPD linear systems with $d$ right hand sides:
\begin{equation}\label{linearsystemtempp}
    (\gamma A^TA+\zeta B^TB)y_i=(\eta A^TA-\xi B^TB)q_i,\ i=1,2,\ldots,d
\end{equation}
for some suitably chosen quadruple $\{\gamma, \zeta, \eta, \xi\}$
by either the sparse Cholesky factorization or the (preconditioned) CG method.

If $\gamma\neq 0$, \eqref{linearsystemtempp} is mathematically equivalent to
$y_i = \bar{y}_i +\frac{\eta}{\gamma} q_i$, where
\begin{equation}\label{linearsystem variant}
    (\gamma A^TA+\zeta B^TB)\bar{y}_i=(-\xi - \frac{\eta\zeta}{\gamma}) B^TB
    q_i,\ i=1,2,\ldots,d.
\end{equation}
Therefore, instead of \eqref{linearsystemtempp},
we prefer to solve the linear systems in \eqref{linearsystem variant}
since we only need to form $B^TBq_i$ and save the cost of computing $A^TAq_i$.
%Similarly, if $\zeta\neq 0$, we can avoid computing $B^TBq_i$.
For these vectors $q_i$,
solving \eqref{linearsystemtempp} or \eqref{linearsystem variant}
has no effect on the accuracy and numerical stability of a direct solver
and, generally, has little effect on the convergence of the CG algorithm.

If $\gamma\neq 0$ and $\zeta\neq 0$,
the linear systems in \eqref{linearsystemtempp} are the normal equations of the generally better conditioned least squares problems
\begin{equation}\label{leastsquaretempp}
    \min \biggl\|\begin{bmatrix}
        \sqrt{\gamma}A\\
       \sqrt{\zeta} B
    \end{bmatrix}y-\begin{bmatrix}
        \frac{\eta}{\sqrt{\gamma}}A\\
        -\frac{\xi}{\sqrt{\zeta}}B
    \end{bmatrix}q_i\biggr\|,
\end{equation}
and we can solve them by the sparse QR factorization or the LSQR algorithm.

Notice that by \Cref{lem:stochastic estimation}, given the sample number $M$,
\Cref{alg:subspace dimension}, i.e.,
the computation of $H_M$, needs the solutions of
$dM$ linear systems or least squares problems; given
the subspace dimension $p$, Step 3 of \Cref{alg:PGSVD}
needs the solutions of $dp$ linear systems or least squares problems
at each iteration, whose total number is $k_cdp$ with $k_c$
the iterations needed for convergence.

Since a large number of linear systems or least squares problems are
to be solved, for the overall efficiency, we definitely
prefer a direct solver whenever the sparse Cholesky or QR factorization is
computationally viable
%, say at the cost of flops $\mathcal{O}(m_1+m_2)\sim
%\mathcal{O}((m_1+m_2)^2)$ or $\mathcal{O}((m_1+m_2)n)$
%with $n$ being the column numbers of $A$ and $B$ and
%$m_1$ and $m_2$ being their row numbers,
since we only need to compute it once and use it repeatedly for solving
more than one problems. It is particularly preferable that $A$ or $B$
is banded and rectangular, for which we take $\{\gamma,\zeta\}=\{1,0\}$
or $\{\gamma,\zeta\}=\{0,1\}$,
so that the resulting coefficient matrix involves only $A$ or $B$;
if both $A$ and $B$ are banded, we take $\{\gamma,\zeta\}=\{1,1\}$,
and the SPD coefficient matrix is then $A^TA+B^TB$. For general sparse
$A$ and $B$, the sparse Cholesky or QR factorization may or may not be
affordable. If its computation and storage is prohibitive,
we have to resort to the CG or LSQR algorithm. In this case,
we are free to choose $\gamma$ and $\zeta$ to obtain
better conditioned SPD linear systems or least squares problems so that the CG
or LSQR converges faster. A natural choice is $\gamma=\zeta=1$, and the
resulting $A^TA+B^TB$ is
well conditioned, provided one of $A$ and $B$ is well conditioned,
which is true in many applications, e.g., linear discrete
ill-posed problems.
%Obviously, a computationally viable (near-) optimal selection
%of $\{\gamma,\zeta\}$ is problem dependent and appealing.
%We leave it as future work.

\section{A convergence analysis}\label{sec: convergence analysis}
Suppose $p\geq n_{gv}$. We partition $X=[X_p,X_{p,\perp}]$, and
set up the following notation:
\begin{align}
    X_p        & =[x_1,\dots, x_p],  \ \ X_{p,\perp}=[x_{p+1},\dots, x_{n}],\label{vdef}        \\
    \Gamma_{p} & =\diag(\gamma_1, \dots, \gamma_p), \ \
    \Gamma_p^{\prime}=\diag(\gamma_{p+1}, \dots, \gamma_{n}),\label{gammadef} \\
    \Sigma_{p} & =\diag(c_1, \dots,c_p), \ \ \Sigma_p^{\prime}=\diag(c_{p+1}, \dots, c_{n}). \label{sigmadef}
\end{align}

Step 3, Step 5 and Step 10 of \Cref{alg:PGSVD} clearly show
\begin{displaymath}
\mathcal{X}^{(k+1)}={\mathrm{span}}\{\hat{X}^{(k+1)}\}
={\mathrm{span}}\{X^{(k+1)}\}
    =P{\mathrm{span}}\{\hat{X}^{(k)}\}.
\end{displaymath}
Inductively, we thus obtain
\begin{equation}\label{Xke}
    \mathcal{X}^{(k)}={\mathrm{span}}\{\hat{X}^{(k)}\}
    =P^k{\mathrm{span}}\{\hat{X}^{(0)}\}=P^k{\mathrm{span}}\{X^{(0)}\}.
\end{equation}
\iffalse
Define
\begin{equation*}
    \bar{P} = H^{1/2}PH^{-1/2}.
\end{equation*}
It can be verified that
\begin{equation*}
    \bar{P} = H^{1/2}X \psi_{d}(2\Sigma_{A}^T\Sigma_{A}-I) (H^{1/2}X)^T
\end{equation*}
is SPSD and similar to $P$
and \eqref{Xke} leads to
\begin{equation}\label{H1/2 Xke}
    {\mathrm{span}}\{H^{1/2}\hat{X}^{(k)}\} = \bar{P}^k {\mathrm{span}}\{H^{1/2}X^{(0)}\}.
\end{equation}
\fi
We next establish convergence results on $\mathcal{X}^{(k)}$ and the
Ritz values $(\hat{c}_i^{(k)}, \hat{s}_i^{(k)})$.

\begin{theorem}\label{Thm:subspace convergence}
    Suppose $\gamma_p>\gamma_{p+1}$
    and $X_p^T H X^{(0)}$ is nonsingular. Then
    \begin{equation}\label{xk}
        \hat{X}^{(k)}=(X_p+X_{p,\perp}E^{(k)})(M^{(k)})^{-\frac{1}{2}}W^{(k)}
    \end{equation}
    with
    \begin{align}
         & E^{(k)}=\Gamma_p^{\prime k}X_{p,\perp}^T H X^{(0)}(X_p^T H X^{(0)})^{-1}\Gamma_{p}^{-k},\label{Ek} \\
         & M^{(k)}=I+(E^{(k)})^TE^{(k)}\label{mkdef}
    \end{align}
    and $W^{(k)}$ being an orthogonal matrix,
    \begin{equation}\label{normek}
        \|E^{(k)}\|\leq \biggl(\frac{\gamma_{p+1}}{\gamma_p}\biggr)^k\|E^{(0)}\|,
    \end{equation}
    and the distance $\epsilon^{(k)}$ between $\mathcal{X}^{(k)}$ and ${\mathrm{span}}\{X_p\}$ (cf. \eqref{C-distdef}) satisfies
    \begin{equation}\label{dist}
        \epsilon^{(k)} ={\mathrm{dist}}(\mathcal{X}^{(k)}, {\mathrm{span}}\{X_p\})= \frac{\|E^{(k)}\|}{\sqrt{1+\|E^{(k)}\|^2}}\leq \biggl(\frac{\gamma_{p+1}}{\gamma_p}\biggr)^k\|E^{(0)}\|.
    \end{equation}
    Label $(\hat{c}_i^{(k)},\hat{s}_i^{(k)}, \hat{u}_i^{(k)},\hat{v}_i^{(k)},\hat{x}_i^{(k)}), i=1,2,\dots,n_{gv}$ such that
    $\hat{c}_{1}^{(k)},\dots, \hat{c}_{n_{gv}}^{(k)}$ are in the same order as $c_1,\dots, c_{n_{gv}}$ in \Cref{Thm:accuracyps},
    and write the Ritz values $\hat{\sigma}_i^{(k)}=\hat{c}_i^{(k)}/\hat{s}_i^{(k)}$.
    Then
    \begin{equation}\label{chord}
        \chi\bigl((\hat{\sigma}_i^{(k)})^2,\sigma_i^2\bigr) \leq
        6(\epsilon^{(k)})^2+2(\epsilon^{(k)})^4, \ i=1,2,\ldots,n_{gv}.
    \end{equation}
\end{theorem}

\begin{proof}
    Expand $X^{(0)}$ as the $H$-orthogonal direct sum of $X_p$ and $X_{p,\perp}$:
    \begin{equation}\label{initial}
        X^{(0)}=X_pX_p^T H X^{(0)}+X_{p,\perp}X_{p,\perp}^T H X^{(0)},
    \end{equation}
    and define
    \begin{equation*}
        E^{(0)} = X_{p,\perp}^T HX^{(0)}(X_p^THX^{(0)})^{-1}.
    \end{equation*}
    Then
    \begin{equation*}
        X^{(0)}(X_p^T HX^{(0)})^{-1}=X_p+X_{p,\perp}E^{(0)}.
    \end{equation*}
    Exploting \eqref{initial}, $PX_p=X_p\Gamma_p$ and $PX_{p,\perp}=X_{p,\perp}\Gamma_p^{\prime}$, we obtain
    \begin{equation}\label{eqsub}
        P^{k}X^{(0)}(X_p^T HX^{(0)})^{-1}\Gamma_p^{-k}=X_p + X_{p,\perp}\Gamma_{p}^{\prime k} E^{(0)}\Gamma_p^{-k}=X_p+X_{p,\perp}E^{(k)}.
    \end{equation}
    with $E^{(k)}$ defined as in \eqref{Ek}. Obviously,
    \begin{equation*}
        \|E^{(k)}\|\leq \biggl(\frac{\gamma_{p+1}}{\gamma_p}\biggr)^k\|E^{(0)}\|\rightarrow 0,
    \end{equation*}
    which is \eqref{normek}.
    By \eqref{Xke} and \eqref{eqsub}, we have
    \begin{equation*}
        {\mathrm{span}}\{\hat{X}^{(k)}\}=P^k{\mathrm{span}}\{X^{(0)}\}
        ={\mathrm{span}}\{X_p+X_{p,\perp}E^{(k)}\}.
    \end{equation*}
    As a result, we can express the $H$-orthonormal $\hat{X}^{(k)}$ as
    \begin{displaymath}
        \hat{X}^{(k)}=(X_p+X_{p,\perp}E^{(k)})(M^{(k)})^{-\frac{1}{2}}W^{(k)},
    \end{displaymath}
    where
    \begin{displaymath}
        M^{(k)}=(X_p+X_{p,\perp}E^{(k)})^{T} H(X_p+X_{p,\perp}E^{(k)})=I+(E^{(k)})^TE^{(k)}
    \end{displaymath}
is the matrix in \eqref{mkdef} and $W^{(k)}$ is an orthogonal matrix.
This establishes \eqref{xk}.

    By the distance definition \eqref{C-distdef}, from \eqref{xk} we have
    \begin{align*}
        \epsilon^{(k)}=\|X_{p,\perp}^T H \hat{X}^{(k)}\|=\|E^{(k)}(M^{(k)})^{-\frac{1}{2}}W^{(k)}\|
        =\frac{\|E^{(k)}\|}{\sqrt{1+\|E^{(k)}\|^2}},
    \end{align*}
    which, together with \eqref{normek}, proves \eqref{dist}.

Exploiting \eqref{eigd} and \eqref{xk}, we obtain
    \begin{align*}
        &\|W^{(k)}(\hat{X}^{(k)})^T A^TA
        \hat{X}^{(k)}(W^{(k)})^T-\Sigma_p^2\|\\
        &=\|(M^{(k)})^{-1/2}(\Sigma_p^2
        +(E^{(k)})^T(\Sigma_{p}^{\prime})^2E^{(k)})(M^{(k)})^{-1/2}-\Sigma_p^2\| \\
        &=\|(M^{(k)})^{-1/2}(\Sigma_p^2
        +(E^{(k)})^T(\Sigma_{p}^{\prime})^2E^{(k)})(M^{(k)})^{-1/2}-\Sigma_p^2\| \\
        & \leq \|(M^{(k)})^{-1/2}\Sigma_p^2(M^{(k)})^{-1/2}-\Sigma_p^2\|+
         \|(M^{(k)})^{-1/2}(E^{(k)})^T(\Sigma_p^{\prime})^2E^{(k)}(M^{(k)})^{-1/2}\|.
    \end{align*}
Let $F^{(k)}=I-(M^{(k)})^{-\frac{1}{2}}.$ Then
    \begin{equation*}
        \|F^{(k)}\|=\|I-(M^{(k)})^{-\frac{1}{2}}\| = 1-\frac{1}{\sqrt{1+\|E^{(k)}\|^2}}\leq \frac{\|E^{(k)}\|^2}{1+\|E^{(k)}\|^2} =(\epsilon^{(k)})^2.
    \end{equation*}
Therefore,
    \begin{align*}
        &\|(M^{(k)})^{-1/2}\Sigma_p^2(M^{(k)})^{-1/2}-\Sigma_p^2\|=
        \|(I-F^{(k)})\Sigma_p^2(1-F^{(k)})-\Sigma_p^2\| \\
        &=\|-\Sigma_p^2F^{(k)}-F^{(k)}\Sigma_p^2+F^{(k)}\Sigma_p^2F^{(k)}\|
        \leq 2(\epsilon^{(k)})^2+(\epsilon^{(k)})^4,
    \end{align*}
which, together with
    \begin{equation*}
        \|(M^{(k)})^{-1/2}(E^{(k)})^T(\Sigma_p^{\prime})^2E^{(k)}(M^{(k)})^{-1/2}\| \leq (\epsilon^{(k)})^2,
    \end{equation*}
establishes
    \begin{equation*}
       \|W^{(k)}(\hat{X}^{(k)})^T A^TA
        \hat{X}^{(k)}(W^{(k)})^T-\Sigma_p^2\| \leq
        3(\epsilon^{(k)})^2+(\epsilon^{(k)})^4.
    \end{equation*}

Since $(\hat{c}_{i}^{(k)})^2,\ i=1,2,\ldots,p$ are the eigenvalues
of $W^{(k)}(\hat{X}^{(k)})^T A^TA \hat{X}^{(k)}(W^{(k)})^T$,
by a perturbation result \cite[Corollary 8.1.6]{golub2013matrix},
from the above relation we obtain
    \begin{equation}\label{alpha}
        |(\hat{c}_{i}^{(k)})^2-c_i^2| \leq  3(\epsilon^{(k)})^2+(\epsilon^{(k)})^4.
    \end{equation}

    On the other hand,
    \begin{align}
       \chi\bigl((\hat{\sigma}_i^{(k)})^2,\sigma_i^2\bigr)&=
        \chi\biggl(\biggl(\frac{\hat{c}_{i}^{(k)}}{\hat{s}_{i}^{(k)}}\biggr)^2,\biggl(\frac{c_i}{s_i}\biggr)^2\biggr)
        \notag\\
        &= \frac{|(c_i^{(k)})^2s_i^2-(s_i^{(k)})^2c_i^2|}{\sqrt{c_i^4+s_i^4}\sqrt{(c_i^{(k)})^4+(s_i^{(k)})^4}}
        \notag\\
        &= \frac{|(c_i^{(k)})^2(1-c_i^2)-(1-(c_i^{(k)})^2)c_i^2|}{\sqrt{c_i^4+s_i^4}\sqrt{(c_i^{(k)})^4+(s_i^{(k)})^4}}
        \notag\\
        &= \frac{|(c_i^{(k)})^2-c_i^2|}{\sqrt{c_i^4+s_i^4}\sqrt{(c_i^{(k)})^4+(s_i^{(k)})^4}}. \label{chi}
    \end{align}
    Since
    \begin{align*}
        &2(c_i^4+s_i^4) \geq (c_i^2+s_i^2)^2=1, \\
        &2((c_i^{(k)})^4+(s_i^{(k)})^4) \geq ((c_i^{(k)})^2+(s_i^{(k)})^2)^2 = 1,
    \end{align*}
    it is from \eqref{chi} that
    \begin{equation}\label{chordandalpha}
        \chi\bigl((\hat{\sigma}_i^{(k)})^2,\sigma_i^2\bigr) \leq 2|(c_i^{(k)})^2-c_i^2|,
    \end{equation}
    which, together with \eqref{alpha}, proves \eqref{chord}.
\end{proof}

The following theorem establishes convergence results on the
left and right Ritz vectors $\hat{u}_i^{(k)},\hat{v}_i^{(k)},\hat{x}_i^{(k)}$
and a new and possibly sharper error bound on $\hat{\sigma}_i^{(k)}$.

\begin{theorem}\label{Thm:quintuples convergence}
    Let
    $\beta^{(k)}=\|P^{(k)}H^{-\frac{1}{2}}A^TAH^{-\frac{1}{2}}(I-P^{(k)})\|$,
    where $P^{(k)}$ is the orthogonal projector onto
    ${\mathrm{span}}\{H^{\frac{1}{2}}\hat{X}^{(k)}\}$.
    Assume that each generalized singular value $\sigma=c/s$ of $(A,B)$ with
    $c \in [c_{\min}, c_{\max}]$ is simple, and define
    \begin{equation*}
        \delta_i^{(k)}= \underset{j\neq i}{\min}|c_i^2-(\hat{c}_j^{(k)})^2|,\
        i=1,2,\ldots,n_{gv}.
    \end{equation*}
    Then for $i=1,2,\ldots,n_{gv}$ we have
    \begin{align}
        \sin\angle(\hat{x}_{i}^{(k)},x_i)_{H} & \leq
        \sqrt{1+\frac{(\beta^{(k)})^2}{(\delta_i^{(k)})^2}}
        \biggl(\frac{\gamma_{p+1}}{\gamma_i}\biggr)^k \|E^{(0)}\|,
        \label{sinex}                \\
        \sin\angle(\hat{u}_{i}^{(k)},u_i) & \leq
        \frac{\|\Sigma_{A}\|}{\hat{c}_{i}^{(k)}}\sin\angle(\hat{x}_{i}^{(k)},x_i)_{H},     \label{sineu}                    \\
        \sin\angle(\hat{v}_{i}^{(k)},v_i) & \leq
        \frac{\|\Sigma_{B}\|}{\hat{s}_{i}^{(k)}}\sin\angle(\hat{x}_{i}^{(k)},x_i)_{H},\label{sinev} \\
        \chi\bigl((\hat{\sigma}_i^{(k)})^2,\sigma_i^2\bigr)&\leq  2\sin^2\angle(\hat{x}_{i}^{(k)},x_i)_{H}. \label{chordcondition}
    \end{align}
\end{theorem}

\begin{proof}
   It is deduced from \eqref{cpp} that $((\hat{c}_i^{(k)})^2,H^{\frac{1}{2}}\hat{x}_{i}^{(k)})$,
    $i=1,2,\ldots,n_{gv}$ are the Ritz pairs of $H^{-\frac{1}{2}}A^TAH^{-\frac{1}{2}}$ with respect to
    ${\mathrm{span}}\{H^{\frac{1}{2}}\hat{X}^{(k)}\}$.
    Notice from \eqref{gsvd} and \eqref{eigd} that
    the eigendecomposition of $H^{-\frac{1}{2}}A^TAH^{-\frac{1}{2}}$ is
    \begin{equation*}
        H^{-\frac{1}{2}}A^TAH^{-\frac{1}{2}}=H^{\frac{1}{2}}X\Sigma_{A}^T\Sigma_{A}X^TH^{\frac{1}{2}},
    \end{equation*}
    whose eigenvector matrix is the orthogonal
    $H^{\frac{1}{2}}X$ and eigenvalues are $c_i^2,\ i=1,2,\dots,n$.
    A direct application of \cite[Theorem 4.6, Proposition 4.5]{saad2011numerical}
    to our case yields
    \begin{align}
        \sin\angle(H^{\frac{1}{2}}\hat{x}_{i}^{(k)},H^{\frac{1}{2}}x_i) &
        \leq \sqrt{1+\frac{(\beta^{(k)})^2}{(\delta_i^{(k)})^2}} \sin\angle(H^{\frac{1}{2}}x_i,
        {\mathrm{span}}\{H^{\frac{1}{2}}\hat{X}^{(k)}\}),\label{ab}      \\
        |(\hat{c}_i^{(k)})^2-c_i^2| &\leq \|H^{-\frac{1}{2}}A^TAH^{-\frac{1}{2}} - c_i^2
        I\|\sin^2\angle(H^{\frac{1}{2}}\hat{x}_{i}^{(k)},H^{\frac{1}{2}}x_i) \nonumber\\
        & \leq \sin^2\angle(H^{\frac{1}{2}}\hat{x}_{i}^{(k)},H^{\frac{1}{2}}x_i).\label{relation}
    \end{align}

   Making use of \eqref{xk}, \eqref{Ek} and \eqref{gammadef}, we obtain
    \begin{align*}
        \sin\angle(H^{\frac{1}{2}}x_i, {\mathrm{span}}\{H^{\frac{1}{2}}\hat{X}^{(k)}\}) &=\sin\angle(H^{\frac{1}{2}}x_i,{\mathrm{span}}\{H^{\frac{1}{2}}\hat{X}^{(k)}(W^{(k)})^T(M^{(k)})^{1/2}\})\\
        &= \sin\angle(H^{\frac{1}{2}}x_i,{\mathrm{span}}\{H^{\frac{1}{2}}(X_p+X_{p,\perp}E^{(k)})\}) \\
        & \leq\sin\angle(H^{\frac{1}{2}}x_i,H^{\frac{1}{2}}(x_i+X_{p,\perp}E^{(k)}e_i))  \\
       % &  \leq\tan\angle(H^{\frac{1}{2}}x_i,H^{\frac{1}{2}}(x_i+X_{p,\perp}E^{(k)}e_i))       %\\
        & \leq \|E^{(k)}e_i\|=\|\Gamma_p^{\prime k} E^{(0)}\Gamma_{p}^{-k}e_i\|
        \leq\|\Gamma_p^{\prime k} E^{(0)}\|\gamma_{i}^{-k}    \\
        & \leq \biggl(\frac{\gamma_{p+1}}{\gamma_i}\biggr)^k \|E^{(0)}\|.
    \end{align*}
    Substituting the last inequality into \eqref{ab} gives
    \begin{equation*}
        \sin\angle(H^{\frac{1}{2}}\hat{x}_{i}^{(k)},H^{\frac{1}{2}}x_i)  \leq \sqrt{1+\frac{(\beta^{(k)})^2}{(\delta_i^{(k)})^2}}\biggl(\frac{\gamma_{p+1}}{\gamma_i}\biggr)^k \|E^{(0)}\|,
    \end{equation*}
    which, combining \eqref{chordandalpha}, \eqref{relation} and
    \begin{equation*}
        \sin\angle(\hat{x}_{i}^{(k)},x_i)_{H}=
        \sin\angle(H^{\frac{1}{2}}\hat{x}_{i}^{(k)},H^{\frac{1}{2}}x_i),
    \end{equation*}
    proves \eqref{sinex} and \eqref{chordcondition}.

    In term of the notations of \Cref{alg:PGSVD}, \eqref{Ax=cu, Bx=sv} reads as
    \begin{equation*}
        A\hat{x}_{i}^{(k)}=\hat{c}_i^{(k)}\hat{u}_{i}^{(k)}, \quad B\hat{x}_{i}^{(k)}=\hat{s}_i^{(k)}\hat{v}_{i}^{(k)}.
    \end{equation*}
    Decompose $\hat{x}_{i}^{(k)}$ into the $H$-orthogonal direct sum:
    \begin{equation*}
        \hat{x}_{i}^{(k)}=x_i\cos\angle(\hat{x}_{i}^{(k)},x_i)_{H}+
        z\sin\angle(\hat{x}_{i}^{(k)},x_i)_{H},
    \end{equation*}
    where $\langle z,x_i\rangle_{H}=0$ and $\|z\|_{H}=1$.
    Write $\phi_i=\angle(\hat{x}_{i}^{(k)},x_i)_{H}$.
    Then from \eqref{gsvdvector} we obtain
    \begin{align}
        \hat{c}_{i}^{(k)}\hat{u}_{i}^{(k)}=A\hat{x}_{i}^{(k)}=A(x_i\cos\phi_i+z\sin\phi_i)=c_i u_i\cos\phi_i+Az\sin\phi_i,\label{uhat}\\
        \hat{s}_i^{(k)}\hat{v}_{i}^{(k)}=B\hat{x}_{i}^{(k)}=B(x_i\cos\phi_i+z\sin\phi_i)=s_i v_i\cos\phi_i+Bz\sin\phi_i.\label{vhat}
    \end{align}
From \eqref{gsvdvector}, since $Az$ and $Bz$ are linear combinations of
    the columns of $U$ and $V$ other than $u_i$ and $v_i$, we have
    \begin{equation*}
        u_i^TAz=0, \         v_i^TBz=0.
    \end{equation*}
    It follows from \eqref{uhat} and \eqref{vhat} that
    \begin{align}
        \sin\angle(\hat{u}_{i}^{(k)},u_i)=\frac{\|Az\|}{\hat{c}_{i}^{(k)}}\sin\angle(\hat{x}_{i}^{(k)},x_i)_{H} \leq \frac{\|\Sigma_{A}\|}{\hat{c}_{i}^{(k)}}\sin\angle(\hat{x}_{i}^{(k)},x_i)_{H}, \label{uhaterror} \\
        \sin\angle(\hat{v}_{i}^{(k)},v_i) = \frac{\|Bz\|}{\hat{s}_{i}^{(k)}}\sin\angle(\hat{x}_{i}^{(k)},x_i)_{H} \leq \frac{\|\Sigma_{B}\|}{\hat{s}_{i}^{(k)}}\sin\angle(\hat{x}_{i}^{(k)},x_i)_{H}, \label{vhaterror}
    \end{align}
    which prove \eqref{sineu} and \eqref{sinev}.
\end{proof}

Note that $\beta^{(k)}\leq\|H^{-\frac{1}{2}}A^TAH^{-\frac{1}{2}}\|\leq 1$ and
\begin{equation*}
    \delta_i^{(k)}= \underset{j\neq i}{\min}|c_i^2-(\hat{c}_j^{(k)})^2|\rightarrow \underset{j\neq
i,j=1,2,\ldots,p}{\min}|c_i^2-c_j^2|
\end{equation*}
for $k\rightarrow\infty$, and keep
in mind that $\|\Sigma_A\|\leq 1$ and $\|\Sigma_B\|\leq 1$.
\Cref{Thm:quintuples convergence} indicates that the CJ-FEAST
GSVDsolver converges,
and the chordal metric errors of Ritz values are approximately the squares of
those of left and right Ritz vectors, whose convergence rates are
$\gamma_{p+1}/\gamma_i,\, i=1,2,\ldots,n_{gv}$.

The requirement $p\geq n_{gv}$ is critical since otherwise
it is known from \Cref{Thm:accuracyps} that
$\gamma_i\approx 1,\ i=1,2,\ldots,p+1$ as $d$ increases
if none of $c_{\min}$ and $c_{\max}$
corresponds to a generalized singular value of
$(A,B)$. In this case, \Cref{Thm:quintuples convergence} indicates
that \Cref{alg:PGSVD} does not converge or stagnates
because of $\gamma_{p+1}/\gamma_i\approx 1,i=1,2,\ldots,p$.

\section{Numerical experiments}\label{sec:numerical experiments}
We report numerical experiments to confirm our theory, and illustrate
the performance of \Cref{alg:subspace iteration}.
All the numerical experiments were performed on an Intel Core i7-9700,
CPU 3.0GHz, 8GB RAM using Matlab R2022b with
$\epsilon_{\mathrm{mach}}=2.22e-16$ under the Microsoft Windows 10 64-bit system.

\cref{tab:Properties of test matrix pairs} lists the test matrix pairs with
some of their properties and the interval $[c_{\min},c_{\max}]$ of interest,
where we use sparse matrices from the SuiteSparse Matrix Collection
\cite{davis2011university} or their transposes as our test matrices $A$. The
matrices $B$ are the $n\times n$ tridiagonal Toeplitz matrix $B_0$
with $3$ and $1$ being the main and off diagonal elements, and
the $(n-1)\times n$ scaled discrete approximation matrix $B_1$ of the first
order derivative operator of dimension one, i.e.,
\begin{equation}\label{B}
    B_0=\begin{bmatrix}
        3 & 1  &  & \\
        1  & \ddots & \ddots & \\
          &  \ddots  &  \ddots    &  1\\
          &     &  1 & 3
    \end{bmatrix} \mbox{\ and \ }
   B_1=\begin{bmatrix}
        1 & -1  &  & \\
          & \ddots & \ddots & \\
          &        &   1    &  -1
    \end{bmatrix}.
\end{equation}

\begin{table}[htbp]
    \begin{center}
        %\resizebox*{\textwidth}{!}{
            \begin{tabular}{|c|c|c|c|c|c|c|}
                \hline
                $A$ & $B$ & $m_1$ & $m_2$   & $n$   & $nnz$ & $[c_{\min},c_{\max}]$ \\
                \hline
                r05$^T$ & $B_0$ & 9690 & 5190 & 5190 & 119713  & $[0.99, 0.995]$ \\  %No. 1758 \sigma 17.43  0.03
                deter4$^T$ & $B_0$ & 9133 & 3235 & 3235 & 28934  & $[0.9, 0.94]$\\   %N0. 1798 \sigma  8.57  5e-3
                lp\_bnl2$^T$ & $B_0$ & 4486 & 2324 & 2324 & 21996 & $[0.9996, 0.9998]$ \\ %No. 605 \sigma 110.2   1.19e-2
                G65 & $B_0$ & 8000 & 8000 & 8000 & 55998 & $[0.7, 0.75]$ \\ % No. 530 \sigma 2.76  8.84e-5
                nopoly & $B_0$ & 10774 & 10774 & 10774 & 103162  &  $[0.9978, 0.9979]$ \\ %No. 442 \sigma 18.3 0
                tomographic1 & $B_0$ & 73159 & 59498 & 59498 & 825987  & $[0.7, 0.75]$ \\ %No. 2264 \sigma 2.79 0
                denormal & $B_0$ & 89400 & 89400 & 89400 & 1424442 & $[0.095, 0.1]$ \\ % ngv =   No. 1892 \sigma 1.29e-1  7.56e-11
                flower\_5\_4$^T$  & $B_1$ & 14721  & 5225 & 5226 & 54392   & $[0.6,0.61]$     \\ %No. 2158
                dw1024$^T$ & $B_1$ & 2048 & 2047 & 2048 & 14208 & $[0.7, 0.8]$  \\ % No. 309
                p05$^T$ & $B_1$ & 9590 & 5089 & 5090 & 69223 & $[0.88, 0.89]$ \\ %No. 1748
                grid2  & $B_1$ & 3296 & 3295 & 3296 & 19454 & $[0.4, 0.45]$ \\ %No. 2418
                \hline
            \end{tabular}
        %}
        \caption{Properties of the test matrix pairs $(A,B)$
        and the interval $[c_{\min},c_{\max}]$ of interest, where $nnz$ is
        the total number of nonzero entries in $A$ and $B$.}
        \label{tab:Properties of test matrix pairs}
    \end{center}
\end{table}

Since $B_0$ is banded and nonsingular, for $(A, B_0)$ we choose
$\{\gamma,\zeta\}= \{0,1\}$ and solve the linear systems
\eqref{linearsystemtempp} by the Cholesky factorization.
For $(A, B_1)$, we choose $\eta = \xi = \gamma = \zeta = 1$,
and use the LSQR algorithm to solve all the least squares
problems in \eqref{leastsquaretemp} accurately with the relative stopping
tolerance $1e-14$.

We claim an approximate GSVD component $(\hat{c},\hat{s},\hat{u},\hat{v},\hat{x})$
to have converged if the residual norm satisfies
\begin{equation}\label{stopcriterion}
    \|r\| \leq (\hat{s}\|A\|_1 + \hat{c}\|B\|_1)\cdot tol = (\hat{s}\|A\|_1 + \hat{c}\|B\|_1)\cdot 1e-8.
\end{equation}

For a practical choice of the series degree $d$,
the analysis and results in \cite{jia2022afeastsvdsolver} is
straightforwardly adapted to the CJ-FEAST GSVDsolver,
and we choose
\begin{equation}\label{dchoice}
    d=\biggl\lceil \frac{D\pi^2}{(\alpha-\beta)^{4/3}}\biggr\rceil -2
\end{equation}
with $\alpha$ and $\beta$ defined in \Cref{Thm:accuracyps} and
$D\in [2,10]$ for \cref{alg:subspace dimension} and $D \in [1,5]$ for
\cref{alg:PGSVD}.

\subsection{Estimating the exact number of desired GSVD components}
For each test matrix pair $(A,B)$, we first compute the `exact' GSVD by
the Matlab built-in function {\sffamily gsvd} to obtain the exact
$n_{gv}$'s. As we have elaborated in \cite{jia2022afeastsvdsolver},
taking $M=20$ or $30$ suffices to obtain a reliable estimate
$H_M$ of $n_{gv}$ for $d$ suitably large, no matter how large $n_{gv}$'s
are. We only report the results on $H_{20}$ by taking $D = 2$ in
\eqref{dchoice}. \cref{table:number estimation} lists the results.
Clearly, $H_{20}$ is an excellent estimate of $n_{gv}$
and $\lceil 1.1 H_M\rceil \geq n_{gv}$ always holds,
showing that the selection strategy \eqref{pchoice} works perfectly.

\begin{table}[htbp]
    \centering
    \begin{tabular}{|c|c|c|c|}
        \hline
        $A$     & $B$ & $n_{gv}$   &  $H_{20}$ \\
        \hline
        r05$^T$ & $B_0$ & 16 & 19.3 \\
        deter4$^T$ & $B_0$ & 74 & 77.1 \\
        lp\_bnl2$^T$ & $B_0$ & 8 & 8.9 \\
        G65 & $B_0$ & 389 & 380.9 \\
        nopoly & $B_0$ & 64 & 64.0 \\
        tomographic1 & $B_0$ & 1572   &   1570.4  \\
        denormal  & $B_0$ & 716   & 714.5   \\
        flower\_5\_4$^T$  & $B_1$ & 18  &  17.7  \\
        dw1024$^T$ & $B_1$ & 105 & 102.1 \\
        p05$^T$ & $B_1$ & 10 & 13.3 \\
        grid2 & $B_1$ & 99 & 98.0 \\
        \hline
    \end{tabular}
    \caption{The exact $n_{gv}$ and its estimate $H_{20}$.}
    \label{table:number estimation}
\end{table}

\subsection{Convergence behavior of the CJ-FEAST GSVDsolver}
Now we illustrate the convergence behavior of \Cref{alg:PGSVD}.

For (r05$^T$, $B_0$), we take $D=4$ to obtain $d = 138$,
and the subspace dimension $p= \lceil 1.2\times  H_{20} \rceil= 24$.
We find that  all the desired approximate GSVD components converged at
$k = 10$ and the most slowly converged generalized singular value is
$0.990310243962983/\sqrt{1-0.990310243962983^2}=7.131066081847035$.
We plot the residuals norms of its Ritz approximations
and the errors of Ritz values in \Cref{fig: convergence process}(a).

For (deter4$^T$, $B_0$), we take $D=2$ to obtain $d = 211$,
and the subspace dimension $p= \lceil 1.1\times  H_{20}\rceil= 85$.
All the desired GSVD components were found at $k = 7$,
and the most slowly converged $\sigma$ is
$0.903826989207309/\sqrt{1-0.903826989207309^2}=2.112248254628957$.
We plot the convergence curves in \Cref{fig: convergence process}(b).

For (dw1024$^T$, $B_1$), we take $D=2$ to obtain $d = 95$,
and the subspace dimension $p= \lceil 1.3\times  H_{20} \rceil= 133$.
\Cref{alg:subspace iteration} converged at $k = 9$
and the most slowly converged $\sigma$
is $0.799007108644669/\sqrt{1-0.799007108644669^2}=1.328751765599029$.
\Cref{fig: convergence process}(c) draws convergence processes.

For (grid2, $B_1$), we take $D=1$ to obtain $d = 185$,
and the subspace dimension $p= \lceil 1.4\times  H_{20} \rceil= 138$.
All the desired approximate GSVD components converged at $k = 10$ and the
most slowly converged $\sigma$ is
$0.400694564042536/\sqrt{1-0.400694564042536^2}=0.437338408922530$. The
convergence curves see \Cref{fig: convergence process}(d).

\begin{figure}[tbhp]
    \centering
    \subfloat[(r05$^T$, $B_0$), the CJ-FEAST GSVDsolver]{\includegraphics[scale=0.43]{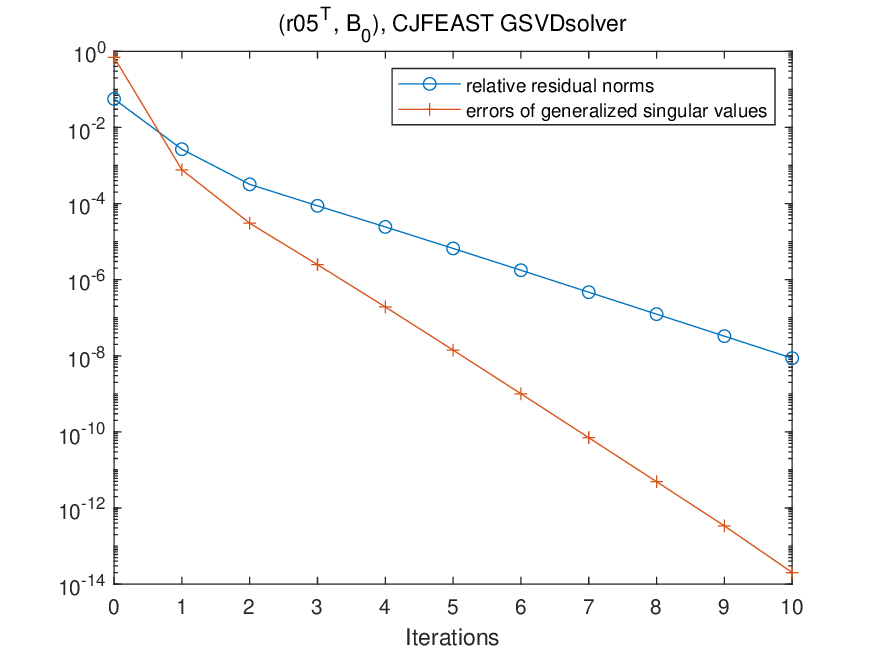}}
    \subfloat[(deter4$^T$, $B_0$), the CJ-FEAST GSVDsolver]{\includegraphics[scale=0.43]{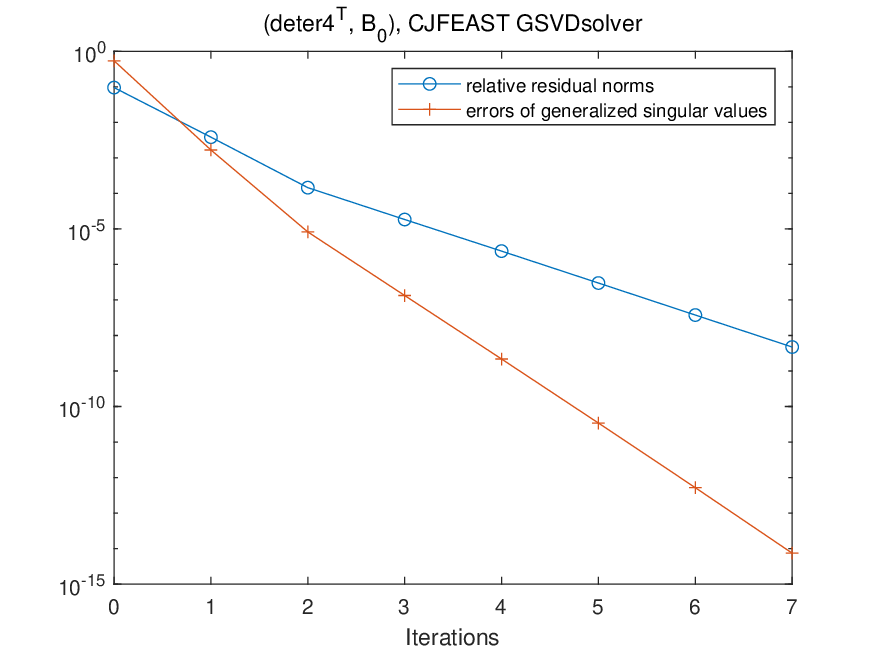}}

    \subfloat[(dw1024$^T$, $B_1$), the CJ-FEAST GSVDsolver]{\includegraphics[scale=0.43]{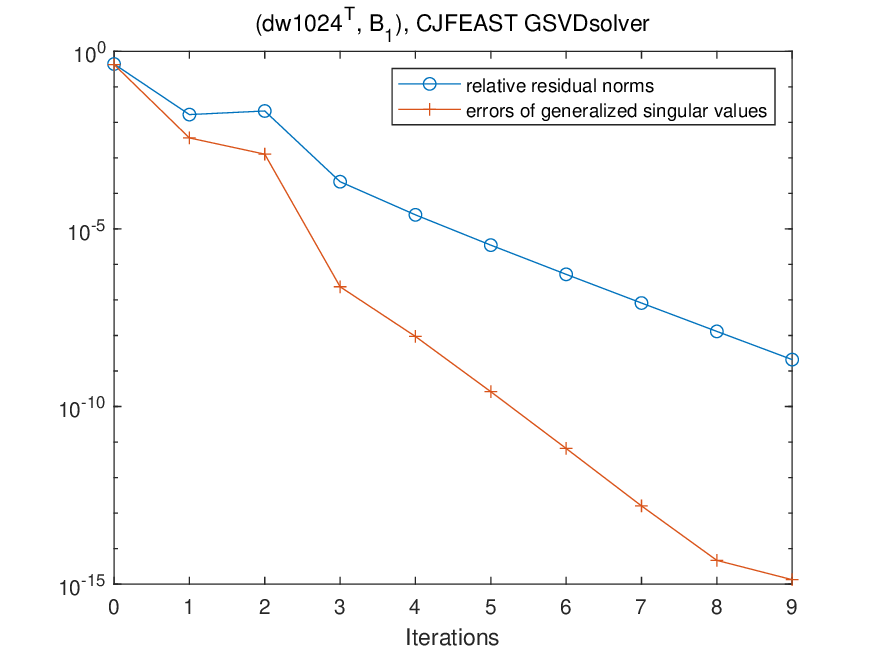}}
    \subfloat[(grid2, $B_1$), the CJ-FEAST GSVDsolver]{\includegraphics[scale=0.43]{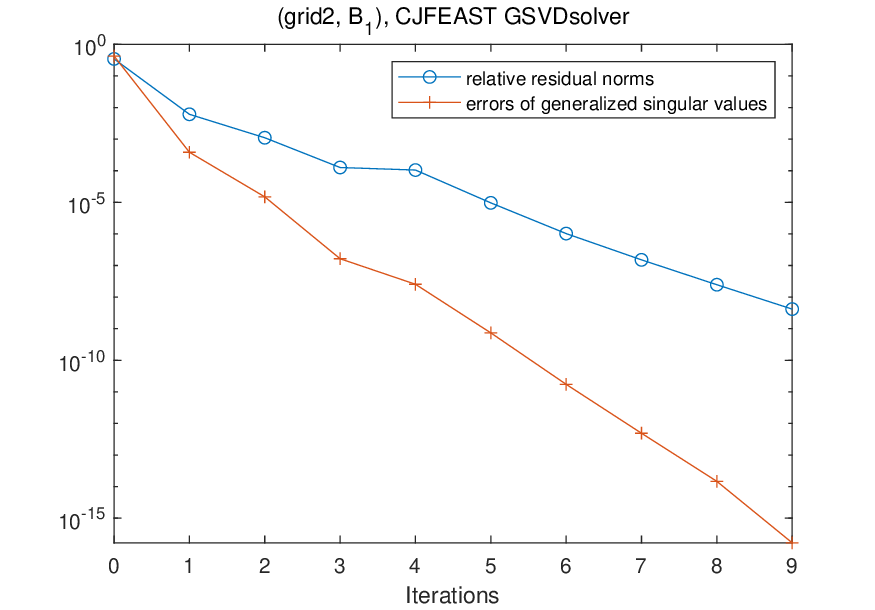}}
    \caption{Convergence processes of Ritz approximations.}
    \label{fig: convergence process}
\end{figure}

The figures clearly demonstrate that the selection
strategy \eqref{pchoice} for the subspace dimension $p$ works well, and the
CJ-FEAST GSVDsolver always converges linearly but quite fast.
In the meantime, the choral metric errors of Ritz values are approximately the
squares of the residual norms,
which confirms the a priori bounds in \Cref{Thm:quintuples convergence}.

\subsection{General performance of the CJ-FEAST GSVDsolver}
We now report numerical results on the test problems in \Cref{tab:Properties of test matrix pairs}. It is seen from \Cref{table:number estimation} that
the numbers $n_{gv}$ of desired GSVD components vary greatly, ranging
from no more than ten to nearly four hundreds. Also, the desired
generalized singular values are either extreme or interior. Therefore,
the degree of difficulty to compute the desired GSVD components
varies widely.

For each test matrix, we take the series degree $d$ as \eqref{dchoice}
using $D=1,2,3$, and select
the subspace dimensions $p=\lceil 1.1H_{20} \rceil$ and $\lceil 1.5H_{20}
\rceil$,
where $H_{20}$ is from \Cref{table:number estimation}.
We then use \Cref{alg:PGSVD} to solve the concerning GSVD problems,
record the total NLS($k$) with $k$ iteration numbers required for
convergence, and list them in \Cref{table:NLS}, where NLS denotes the
total number of linear systems \eqref{linearsystemtempp}
that are solved and can be regarded as a measure of overall
efficiency. For each problem and the same $p$,
we used the same initial $Q^{(0)}$, which was the $Q$-factor of the
QR factorization of a matrix generated randomly
in a normal distribution. \Cref{table:NLS} lists the results obtained.

\begin{table}[htbp]
    \centering
    \begin{tabular}{|c|c|c|c|c|}
    \hline
    \multirow{2}{*}{$(A,B)$} & \multirow{2}{*}{$p$} & \multicolumn{3}{c|}{NLS($k$)}  \\ \cline{3-5}
                      &                   & $D=1$ & $D=2$ & $D=3$ \\ \hline
    \multirow{2}{*}{(r05$^T$, $B_0$)} &   22                & 30492(42) & 31416(21) & 31724(14) \\
                      &    29               & 40194(42) & 41412(21) & 38831(13) \\ \hline
    \multirow{2}{*}{(deter4$^T$, $B_0$)} &  85                 & 142800(16) & 125545(7) & 108120(4) \\
                      &   116                & 133980(11) & 122380(5) & 110664(3) \\ \hline
    \multirow{2}{*}{(lp\_bnl2$^T$, $B_0$)} &    10               & 2000(4) & 3060(3) & 3080(2) \\
                      &  14                 & 2100(3) & 4284(3) & 4312(2) \\ \hline
    \multirow{2}{*}{(G65, $B_0$)} &  419       &  5703428(83) & 5530800(40) & 4977720(24) \\
                      &    572               & 938080(10) & 755040(4) & 1132560(4) \\ \hline
    \multirow{2}{*}{(nopoly, $B_0$)} &  71       & 796620(33) & 726330(15) & 654336(9) \\
                      &    96    & 163200(5) & 196416(3) & 294912(3) \\ \hline
    \multirow{2}{*}{(tomographic1, $B_0$)} & 1728 & 15583104(55) & 15629760(28) & 13934592(17) \\
                      & 2356   & 3541068(10) & 3946300(6) & 4749696(5) \\ \hline
    \multirow{2}{*}{(denormal, $B_0$)} & 786 & 6277782(50) & 5413968(22) & 4649976(13) \\
                      & 1072   & 1397888(9) & 1406464(5) & 1585488(4) \\ \hline
    \multirow{2}{*}{(flower\_5\_4$^T$, $B_1$)} &   20   & 375200(14) & 429120(8) & 402300(5) \\
                      &  27 & 108540(3) & 217242(3) & 217242(2) \\ \hline
    \multirow{2}{*}{(dw1024$^T$, $B_1$)} &   113                & 467368(88) & 461605(43) & 452452(28) \\
                      &    154              & 101332(14) & 87780(6) & 88088(4) \\ \hline
    \multirow{2}{*}{(p05$^T$, $B_1$)} &  15                 & 98100(10) & 117900(6) & 117960(4) \\
                      &   20                & 104640(8) & 104800(4) & 117960(3) \\ \hline
    \multirow{2}{*}{(grid2, $B_1$)} &  108  & 779220(39) & 681156(17) & 601560(10) \\
                      & 147   & 190365(7) & 218148(4) & 245637(3) \\ \hline
    \end{tabular}
    \caption{The results of different $d$ and $p$: NLS and the number $k$
    of iterations used for convergence}
    \label{table:NLS}
\end{table}

We make some comments on \Cref{table:NLS}.
(i) For each problem and the same subspace dimension $p$,
the bigger is the series degree $d$, the fewer iterations are used.
(ii) For each problem and the same series $d$, the bigger is
$p$, the fewer iterations are used; for at least half of the test
problems, the bigger $p$
can speed up the solver very substantially, as is indicated by NLS's
and $k$'s. (iii) For most of the problems, the NLS($k$) do not
change much for the same $p$ and different $d$'s, indicating
that, for the same $p>n_{gv}$, the overall efficiency
is not sensitive to $d$ and a suitably large $d$ works well.
(iv) The bigger $n_{gv}$, the harder it is to solve the problem
because of relatively big $p$'s.
%(v) The selection strategies \eqref{dchoice} for $d$
%and \eqref{pchoice} for $p$ work well.

%The results and above comments suggest that in general one should take
%$p>n_{gv}$ considerably in order to achieve fast convergence.

\section{Conclusions}\label{sec: conclusion}

We have proposed a general projection method for computing
a partial GSVD of a large regular matrix pair $(A,B)$,
and particularly presented a CJ-FEAST GSVDsolver
for the computation of the GSVD components of with
the generalized singular values in a given interval.
%The new solver is a nontrivial extension of the FEAST eigensolver for the
%generalized eigenvalue problem and has a number of
%distinctive mathematical and algorithmic peculiarities.
The projection method works on
the GSVD problem of $(A,B)$ directly and is mathematically
equivalent to the Rayleigh--Ritz projection of the SPD pair
$(A^TA,H)$ with a given right subspace. The CJ-FEAST
GSVDsolver constructs an approximate spectral projector $P$ of $(A^TA,H)$
corresponding to the generalized singular values of interest
by the CJ series expansion rather than a contour integral-based
numerical quadrature or rational filtering that needs to solve several, i.e.,
$\frac{nnodes}{2}p$, shifted and indefinite
large linear systems at each iteration, where the even $nnodes$ is
the number of nodes of an underlying numerical quadrature.
The CJ-FEAST GSVDsolver exploits subspace iteration on $P$ to
generate a sequence of left and right subspaces,
projects the concerning GSVD problem of $(A,B)$ directly onto the
left and right subspaces, and computes the Ritz approximations to
the desired GSVD components.
%Unlike the contour integral
%based approaches and numerical quadrature, we construct the approximate
%spectral projector $P$ by the CJ series expansion without solving
%large shifted linear systems at each iteration.

%We have analyzed the convergence of the CJ-FEAST GSVDsolver.
In terms of the CJ series degree $d$, we have established
a reliable estimate for the number $n_{gv}$ of desired GSVD components
and the accuracy estimates for the approximate spectral projector $P$ and its
eigenvalues. For the CJ-FEAST GSVDsolver, we have established a number of
convergence results on the right searching subspace $\mathcal{X}^{(k)}$ and
the Ritz approximations. Based on
some of the results obtained, we have proposed practical
selection strategies for the series degree $d$
and subspace dimension $p$, and developed the
CJ-FEAST GSVDsolver.

Illuminating numerical experiments have confirmed our theoretical
results and analysis, and demonstrated that the CJ-FEAST GSVDsolver
is practical.

%\bibliographystyle{siamplain}
%\bibliography{references}

\end{document}